\documentclass[12pt]{amsart}
\usepackage{graphicx,cleveref}
\usepackage{fullpage}
\usepackage{amssymb}
\usepackage{url}
\usepackage{color}
\usepackage{float}
\usepackage{caption}
\usepackage{tikz,tkz-graph}
\title{(Generalized) binomial edge ideals and their regularity}
\author{A. V. Jayanthan}
\address{Department of Mathematics, Indian Institute of Technology Madras, Chennai, Tamil Nadu, India - 600036.}
\email{jayanav@iitm.ac.in}

\author{Arvind Kumar}
\address{Department of Mathematical Sciences, New Mexico State University, 1305 Frenger St, Las Cruces, NM 88001, USA}
\email{arvkumar@nmsu.edu}
\urladdr{https://sites.google.com/view/arvkumar/home}

\newcommand{\reg}{\operatorname{reg}}

\newcommand{\iv}{\operatorname{iv}}
\newcommand{\ini}{\operatorname{in}}

\newcommand{\GG}{\mathcal{G}}

\newcommand{\NN}{\mathbb{N}}

\newcommand \ma{\operatorname{m}}

\newtheorem{theorem}{Theorem}[section]
\newtheorem{lemma}[theorem]{Lemma}
\newtheorem{proposition}[theorem]{Proposition}
\newtheorem{corollary}[theorem]{Corollary}
\newtheorem{remark}[theorem]{Remark}
\newtheorem{definition}[theorem]{Definition}
\newtheorem{example}[theorem]{Example}

\newtheorem{conjecture}[theorem]{Conjecture}
\newtheorem{question}[theorem]{Question}

\begin{document}

\begin{abstract}
In this article, we survey the recent results on the Castelnuovo-Mumford regularity of binomial edge ideals and generalized binomial edge ideals. We also generalize some of the known upper bounds for binomial edge ideals to the case of generalized binomial edge ideals.
\end{abstract}
\maketitle

\section{Introduction}
Ever since Stanley used commutative algebraic techniques to give a solution to the Anand-Dumir-Gupta conjecture \cite{Stan73} and a solution to upper bound conjecture for the case of triangulations of spheres \cite{Stan75}, associating ideals to combinatorial objects and understanding the interplay between their algebraic and combinatorial properties have been an active direction of research. Stanley-Reisner rings of simplicial complexes and (monomial) edge ideals of finite simple graphs are two such associations that have been extensively studied in the past few decades. 
Another one, which is the main object of this article, is called \textit{binomial edge ideals}. This was introduced by Herzog et al. \cite{HHHKR}, and independently by Ohtani \cite{Oh11}. Let $G$ be a simple graph with the vertex set $V(G) = [n] = \{1,\ldots,n\}$ and edge set $E(G)$. Then the binomial edge ideal of $G$ is defined to be \[J_G := ( x_iy_j - x_jy_i ~: \{i,j\} \in E(G) ) \subset S=K[x_1,\ldots, x_n, y_1,\ldots, y_n],\]
where $K$ is a field.
Ohtani and Herzog et al. described several basic properties of the binomial edge ideals such as a primary decomposition, a Gr\"obner basis, height etc. In \cite{HHHKR}, the authors related the binomial edge ideals to some specific case of a conditional independence statements arising in algebraic statistics.

Understanding the algebraic properties and invariants of the binomial edge ideals in terms of the structure and combinatorial invariants of the corresponding graph has been an active direction of research immediately after the introduction of this concept. In this article, we focus on the study of \textit{Castelnuovo-Mumford regularity}, or simply referred to as \textit{regularity} of binomial edge ideals. Regularity measures the computational complexity of a graded ideal or module. While it is defined using homological terms, its connection to other algebraic and geometric properties and invariants makes it an important object of study in the literature. 

In the context of binomial edge ideals, the efforts are to compute the regularity of a binomial edge ideal in terms of the combinatorial invariants associated to the graph. In general, it is extremely difficult to construct a minimal free resolution, and thus obtaining an explicit formula for the regularity in terms of the data associated with the graph becomes a challenging task. Therefore, researchers have been looking for efficient upper and/or lower bounds for the regularity. In this article, we survey the results on generic upper and lower bounds as well as some precise expressions for specific classes of graphs. We also prove some new results, generalizing some upper bounds known in the case of binomial edge ideals to the case of generalized binomial edge ideals.

Our paper is organized as follows: In the second section, we gather the graph theoretic and algebraic definitions, notation and some basic results required for the rest of the paper. In \Cref{sec:upper-bound}, we survey results on the upper bounds for the regularity of binomial edge ideals known in the literature. Most of these results are on binomial edge ideals. We generalize some of these upper bounds to the case of generalized binomial edge ideals. We focus on the upper bounds for specific classes of graphs in \Cref{sec:ub-specific}. We also provide a simple proof of a conjecture of Matsuda and Murai, which was proved by Kiani and Saeedi Madani, using a new invariant we introduced in the previous section. The last section is devoted to survey the results on lower bounds for the regularity of binomial edge ideals.

\section{Preliminaries}
In this section, we recall the necessary notation and terminology required for the rest of the paper.

\subsection{Basics on Graph Theory}
A graph is $G$ represented by $V(G)$, a set of vertices, and $E(G)$, a set of edges. For a graph $G$, we will denote the number of vertices by $n(G)$, number of edges in $G$ by $e(G)$ and the number of components of $G$ by $c(G)$. Throughout this paper, all graphs we consider are finite simple graphs, i.e., graphs having no multiple edges between a pair of vertices and no loops at any vertex. We denote by $\mathcal{G}$, the set of all finite simple graphs. Given a vertex $v$ of $G$, let $N_G(v) := \{u \in V(G) : \{u,v\} \in E(G)\}$. We say $u$ and $v$ are adjacent if $\{u,v\} \in E(G)$. For $v \in V(G)$, the degree of $v$ in $G$, denoted by $\deg_G(v)$, is the cardinality of $N_G(v)$.

A graph is said to be a \textit{complete graph} if each pair of its vertices is connected by an edge. A complete graph on $n$ vertices is denoted by $K_n$. A subset $S$ of $V(G)$ is said to be an independent set if there are no edges in $G$ between the vertices of $S$. A graph is said to \textit{bipartite} if one can write $V(G) = X \sqcup Y$ such that both $X$ and $Y$ are independent sets. A bipartite graph is a \textit{complete bipartite} graph if each vertex of $X$ is adjacent to all vertices of $Y$. We denote by $K_{m,n}$ a complete bipartite graph with $|X|=m$ and $|Y|=n$.


For a subset $X$ of $V(G)$, an \textit{induced subgraph} on $X$, denoted by $G[X]$, is the graph with vertex set $X$ and those edges in $G$ connecting the vertices of $X$. A set of vertices \(\mathcal{V} \subseteq V(G)\) is called a \textit{clique} of a graph \(G\) if the induced subgraph on \(\mathcal{V}\) forms a complete graph. A {\it maximal clique} of \(G\) is a clique that is not contained in any larger clique of \(G\). We denote the number of maximal cliques in the graph \(G\) as \(\mathcal{C}(G)\). The maximum size of a maximal clique in $G$ is called the \textit{clique number} of $G$, denoted by $\omega(G)$. A vertex is said to be an \textit{internal vertex} if it is contained in at least two maximal cliques. Number of internal vertices of a graph $G$ is denoted by $\iv(G)$.

A graph $G$ is a \textit{cycle graph} if $\deg_G u = 2$ for all $u \in V(G)$. If $G$ is a cycle graph on $n$ vertices, then one can label the vertices from $1$ to $n$ so that edge set is given by $E(G) = \{\{i,i+1\} : 1\leq i < n\} \cup \{\{1,n\}\}$. A graph is said to be \textit{chordal} if it has no induced cycle of length $4$ or more. A vertex $v$ of $G$ is said to be a \textit{cut vertex} if the induced subgraph on $V(G) \setminus \{v\}$ has more connected components than that of $G$. A \textit{block} of a graph $G$ is a maximal induced subgraph having no cut vertex. A \textit{block graph} is a graph with all its blocks being complete graphs. It is easy to observe that block graphs are chordal graphs. A tree is a graph having no cycles as subgraphs. 


\begin{figure}[H]
\begin{tikzpicture}
\draw (2,3)-- (1,2);
\draw (1,2)-- (2,1);
\draw (2,1)-- (3,2);
\draw (3,2)-- (2,3);
\draw (2,3)-- (2,1);
\draw (1,2)-- (3,2);
\draw (2,3)-- (3,3);
\draw (3,3)-- (3,2);
\draw (3,2)-- (4,2);
\draw (4,2)-- (3.5,3);
\draw (4,2)-- (4.5,3);
\draw (4,2)-- (5,2.5);
\draw (4,2)-- (5,1.5);
\draw (5,1.5)-- (5,2.5);
\draw (5,2.5)-- (6,2);
\draw (6,2)-- (5,1.5);
\begin{scriptsize}
\fill [color=black] (2,3) circle (1.5pt);
\draw[color=black] (2.08,3.2) node {$1$};
\fill [color=black] (1,2) circle (1.5pt);
\draw[color=black] (1,1.71) node {$2$};
\fill [color=black] (2,1) circle (1.5pt);
\draw[color=black] (2.08, 0.8) node {$3$};
\fill [color=black] (3,2) circle (1.5pt);
\draw[color=black] (3.07,1.71) node {$4$};
\fill [color=black] (3,3) circle (1.5pt);
\draw[color=black] (3.07,3.2) node {$5$};
\fill [color=black] (4,2) circle (1.5pt);
\draw[color=black] (4.07,1.71) node {$6$};
\fill [color=black] (3.5,3) circle (1.5pt);
\draw[color=black] (3.57,3.2) node {$7$};
\fill [color=black] (4.5,3) circle (1.5pt);
\draw[color=black] (4.57,3.2) node {$8$};
\fill [color=black] (5,2.5) circle (1.5pt);
\draw[color=black] (5.04,2.71) node {$9$};
\fill [color=black] (5,1.5) circle (1.5pt);
\draw[color=black] (5.04,1.3) node {$10$};
\fill [color=black] (6,2) circle (1.5pt);
\draw[color=black] (6.2,2.11) node {$11$};
\end{scriptsize}
\end{tikzpicture}
\caption{A graph $G$}
\label{example:gen-graph}
\end{figure}
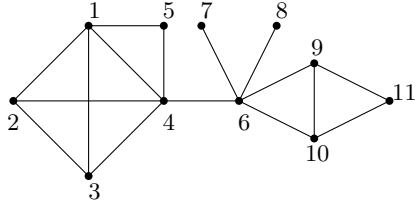

\begin{figure}[H]
\begin{tikzpicture}[scale=2]
\draw (1,2)-- (0.5,1.5);
\draw (0.5,1.5)-- (1,1);
\draw (1,1)-- (1.5,1.5);
\draw (1.5,1.5)-- (1,2);
\draw (1.5,1.5)-- (2,1.5);
\draw (3,1.5)-- (2.5,2);
\draw (3,1.5)-- (3.5,2);
\draw (3.5,2)-- (4,1.5);
\draw (4,1.5)-- (3.5,1);
\draw (3.5,1)-- (3,1.5);
\draw (3.5,2) -- (3.5, 1);
\begin{scriptsize}
\fill [color=black] (1,2) circle (0.8pt);
\draw[color=black] (1.06,2.1) node {$1$};
\fill [color=black] (0.5,1.5) circle (0.8pt);
\draw[color=black] (0.37,1.6) node {$2$};
\fill [color=black] (1,1) circle (0.8pt);
\draw[color=black] (1.1,0.9) node {$3$};
\fill [color=black] (1.5,1.5) circle (0.8pt);
\draw[color=black] (1.56,1.6) node {$4$};
\fill [color=black] (2,1.5) circle (0.8pt);
\draw[color=black] (2.07,1.6) node {$6$};
\fill [color=black] (3,1.5) circle (0.8pt);
\draw[color=black] (3.02,1.65) node {$6$};
\fill [color=black] (2.5,2) circle (0.8pt);
\draw[color=black] (2.57,2.1) node {$7$};
\fill [color=black] (3.5,2) circle (0.8pt);
\draw[color=black] (3.56,2.13) node {$9$};
\fill [color=black] (4,1.5) circle (0.8pt);
\draw[color=black] (4.1,1.6) node {$11$};
\fill [color=black] (3.5,1) circle (0.8pt);
\draw[color=black] (3.6,0.9) node {$10$};
\end{scriptsize}
\end{tikzpicture}
\caption{$G_1$ (left) \& $G_2$ (right)}
\label{example:ind-sub-graph}
\end{figure}
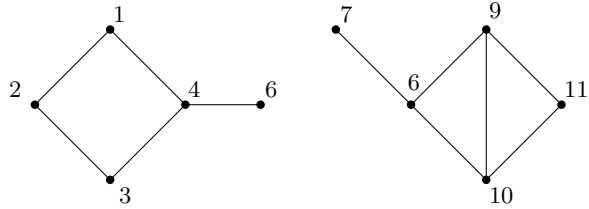

Using the graphs given in \Cref{example:gen-graph} and \Cref{example:ind-sub-graph}, we illustrate some of the concepts and invariants introduced in the previous paragraph. To begin with, note that the graph $G_1$ in \Cref{example:ind-sub-graph} is not an induced subgraph of $G$, since the edges $\{2,4\}$ and $\{1,3\}$ are not present in $G_1$. The graph $G_2$ is an induced subgraph of $G$. In $G$, $\{1,2,3\}$ is a clique, but not a maximal clique and $\{1,2,3,4\}$ is a maximal clique. Note that $\{1,4,5\}$ and $\{4,6\}$ are also maximal cliques. Note that $N_G(1) = \{2,3,4,5\}$ so that $\deg_G(1) = 4$. It can be seen that $G$ is not a block graph since $\{6,9,10,11\}$ is a block, but it is not a clique. In $G$, $\{1,4,6,9,10\}$ are internal vertices and $\{2,3,5,7,8,11\}$ are free vertices.  Hence $\iv(G) = 5$. Maximum size of a maximal clique of $G$ is $4$ and there are $7$ maximal cliques in $G$. Hence $\omega(G) = 4$ and $\mathcal{C}(G)= 7$.

A graph on the vertex set $[n]$ is said to be a \textit{closed graph} if there exists a labelling of its vertices so that for $i < k$ and $\{i,j\} \in E(G)$, $\{i,j\}, \{j, k\}\in E(G)$ for each $i < j < k$. It was shown by Herzog et al. that the closed graphs are precisely the graphs whose binomial edge ideals have a quadratic Gr\"obner basis, \cite{HHHKR}. One may observe that being closed graph is dependent on the labelling of the vertices. For example, a path $G$ on $4$ vertices with  $E(G) =\{\{3,1\},\{1,2\},\{2,4\}\}$ is not closed since $\{2,4\} \in E(G),$ but $\{3,4\} \notin E(G)$. If we relabel the vertices so that $E(G) = \{\{1,2\},\{2,3\},\{3,4\}\}$, then $G$ is closed.


\subsection{Algebraic basics:} We have seen the definition of binomial edge ideals in the introduction section. Rauh, in \cite{Rauh13}, generalized this concept of binomial edge ideals to define \textit{generalized binomial edge ideals}. For a graph $G$ on the vertex set $[n]$ and a fixed integer $m\geq 2$, the generalized binomial edge ideal associated to $G$, denoted by $J_{K_m,G} \subset K[x_{ij} : i \in [m] \text{ and } j \in [n]]$, is the ideal generated by the binomials $x_{ik}x_{jl} - x_{il}x_{jk}$ for all $1\leq i, j \leq m$ and $\{k,l\} \in E(G)$. 

For the graph $G_1$ given in \Cref{example:gen-graph}, the binomial edge ideal is given by 
\begin{eqnarray*}
J_{G_1} & = &(x_1y_2-x_2y_1, x_1y_4-x_4y_1, x_2y_3-x_3y_2, x_3y_4-x_4y_3, x_4y_6-x_6y_4)\\ & \subset & K[x_1,x_2,x_3,x_4,x_6,y_1,y_2,y_3,y_4,y_6].
\end{eqnarray*}
And the generalized binomial edge ideal, for $m = 3$, is given by 
\begin{eqnarray*}
J_{K_3,G_1} & = &(x_{11}x_{22}-x_{12}x_{21}, x_{11}x_{24}-x_{14}x_{21}, x_{12}x_{23}-x_{13}x_{22}, x_{13}x_{24}-x_{14}x_{23}, x_{14}x_{26}-x_{16}x_{24}, \\
& & x_{11}x_{32}-x_{12}x_{31}, x_{11}x_{34}-x_{14}x_{31}, x_{12}x_{33}-x_{13}x_{32}, x_{13}x_{34}-x_{14}x_{33}, x_{14}x_{36}-x_{16}x_{34},\\
& & x_{21}x_{32}-x_{22}x_{31}, x_{21}x_{34}-x_{24}x_{31}, x_{22}x_{33}-x_{23}x_{32}, x_{23}x_{34}-x_{24}x_{33}, x_{24}x_{36}-x_{26}x_{34}) \\
& \subset  & K[x_{ij} : i \in \{1,2,3\}, j  \in \{1,2,3,4,6\}].
\end{eqnarray*}

Let $R = K[t_1,\ldots,t_n]$ be a polynomial ring on $n$ variables. We always consider the standard grading on the polynomial rings considered in this article, i.e., $\deg  t_i = 1$ for all $i=1,\ldots, n$. Let $M=\oplus_{d\geq 0}M_d$ be a finitely generated graded $S$-module. Let 
\[
0 \to F_p \overset{d_p}{\to} F_{p-1} \overset{d_{p-1}}{\to} \cdots \to F_1 \overset{d_1}{\to} F_0 \overset{d_0}{\to} M \to 0
\]
be the minimal free resolution of $M$. Write $F_i = \oplus_{j \in \mathbb{Z}} R(-j)^{\beta_{ij}(M)}$ for $i=0,\ldots,p$. From the minimal free resolution, one can read off several important invariants. Here we define the Castelnuovo-Mumford regularity:
\[
\reg(M) := \max\{j-i  : \beta_{ij}(M) \neq 0\}.
\]
The following is one of the most basic properties of the regularity, but probably the most useful in investigating the regularity of binomial edge ideals:

\begin{proposition}\cite[Proposition 18.6]{Peeva11}\label{reg-lemma}
Suppose that $0 \to U \to U' \to U'' \to 0$ is a short exact sequence of finitely generated graded $R$-modules with graded homomorphisms of degree $0$. Then,
\begin{enumerate}
    \item If $\reg(U') > \reg(U''),$ then $\reg(U) = \reg(U').$
    \item If $\reg(U') < \reg(U'')$, then $\reg(U) = \reg(U'')+1$.
    \item If $\reg(U') = \reg(U'')$, then $\reg(U) \leq  \reg(U'')+1$.
\end{enumerate}
\end{proposition}

Another useful result regarding the resolution is the following:
\begin{proposition}\label{tensor-res}
Let $I \subset R_1 = K[x_1, \ldots, x_n]$ and $J \subset R_2 = K[y_1,\ldots,y_n]$ be homogeneous ideals. Let $I+J$ be the ideal generated by $I\cup J$ in $R = K[x_1,\ldots,x_n,y_1,\ldots,y_m] = R_1 \otimes_K R_2$. Let $\mathbf{F_{\bullet}}$ be a minimal free resolution of $R_1/I$ and $\mathbf{G_{\bullet}}$ be a minimal free resolution of $R_2/J$. Then $\mathbf{F_{\bullet}} \otimes \mathbf{G_{\bullet}}$ is a minimal free resolution of $R/I+J$.
\end{proposition}

The consequence of this result is that one can obtain the graded Betti numbers of $I+J$ from the graded Betti numbers of $I$ and that of $J$. In particular, one can see that $\reg(R/I+J) = \reg(R_1/I) + \reg(R_2/J)$.

Throughout this article, $S$ will always denote the polynomial ring associated with the generalized binomial edge ideal $J_{K_m,G}$ of a given graph $G$. We will use different notations if there is more than one graph in a given context.


\section{General Regularity Upper bounds}\label{sec:upper-bound}
In this section, we review prominent general upper bounds on regularity found in the literature. In some cases, we will provide concise proofs and more generalized results that establish regularity upper bounds for binomial edge ideals using the novel inductive method developed in \cite{Kumar-reg}.

The investigation into the regularity of binomial edge ideals began shortly after the concept was introduced in \cite{HHHKR} and \cite{Oh11}. This early research was primarily conducted by Saeedi Madani and Kiani \cite{KMEJC}, who focused on the linearity of the resolution of binomial edge ideals. They established an upper bound for the regularity of closed graphs. 

\begin{theorem}\cite[Theorem 3.2]{KMEJC}
Let $G \in \GG$ be a closed graph. Then, $\reg(S/J_G) \le \mathcal{C}(G).$
\end{theorem}

Shortly thereafter, Matsuda and Murai \cite{MM} established the general upper and lower bounds for the regularity of binomial edge ideals. Their work introduced the first known general upper bound in the literature, presented as follows \cite{MM}:

\begin{theorem}\cite[Theorem 1.1]{MM} \label{first-gen-up}
Let $G \in \GG$. Then, $\reg(S/J_G) \le n(G)-c(G).$\footnote{The original statement by Matsuda and Murai \cite[Theorem 1.1]{MM} asserts that \(\reg(S/J_G) \leq n(G) - 1\). If \(G = \sqcup_{i=1}^{c(G)} G_i\), then \(S/J_G \cong S_1/J_{G_1} \otimes_{K} \cdots \otimes_{K} S_{c(G)}/J_{G_{c(G)}}\). Consequently, it follows that \(\reg(S/J_G) = \sum_{i=1}^{c(G)} \reg(S_i/J_{G_i})\), where \(S_i = K[x_j, y_j : j \in V(G_i)]\). 

Thus, \Cref{first-gen-up} can be easily derived by applying Matsuda and Murai's result to each component. We present this approach, as later in the section, we plan to provide a concise and alternative proof of the result using the unified concept of a compatible map.}
\end{theorem}

Using conditions from Lyubeznik resolutions, Matsuda and Murai demonstrated that the regularity of the initial binomial edge ideal of a connected graph with respect to lexicographic order is bounded above by \( n(G) - 1 \). By applying the well-known result \(\reg(S/I) \le \reg(S/\text{in}_{\text{lex}}(I))\), they established this bound on regularity. Saeedi Madani and Kiani also adopted an initial ideal approach to derive their results. 

Matsuda and Murai, in their paper,   conjectured that this upper bound is attained only when \( G \) is a path graph on \([n]\), see \cite[Conjecture 3.10]{MM}.

At the same time, these ideals were generalized to a more general class of ideals by Rauh in \cite{Rauh13}, known as generalized binomial edge ideals, and then by Ene et al. \cite{EHHQ} to binomial edge ideals of a pair of graphs. 
Soon after these generalizations, Saeedi Madani and Kiani took a stab at the regularity of these ideals and proved the generalized version of their own result for the generalized binomial edge ideal of closed graphs. They proved the following: 
\begin{theorem}\cite[Theorem 13]{MK13-pair}
Let $G \in \GG$ be a closed graph. Then, for any positive integer $m\ge 2$, $$\reg \left(S/J_{K_m,G}\right) \le \min \left\{ { m \choose 2} \mathcal{C}(G), ~e(G)\right\}.$$
\end{theorem}

Saeedi Madani and Kiani raised a question about the upper bound for the regularity of generalized binomial edge ideals in their article. This question has since been frequently cited as a conjecture in the literature on binomial edge ideals, states the following:

\begin{question} \label{reg-que}\cite[Question]{MK13-pair}
Let $G \in \GG$. Then, for any $m\ge 2$, $$\reg \left(S/J_{K_m,G}\right) \le \min \left\{ { m \choose 2} \mathcal{C}(G), ~e(G)\right\}.$$
\end{question}

Significant progress was made on the aforementioned question regarding the binomial edge ideal, particularly for the case where \( m = 2 \) and various classes of graphs. Ene and Zarojanu \cite{EZ} were the first to initiate progress on this conjecture, proving it for block graphs. We advanced the conjecture in our research \cite{JA1}, where we established it for fan graphs, a subclass of chordal graphs. Much of this research relies on an inductive technique that, until 2019, had been applicable to only a few classes of graphs.

In 2020, Rouzbahani Malayeri, Saeedi Madani, and Kiani \cite{RSK21}, and independently Kumar \cite{AR3}, proved the conjecture for the case when \( m = 2 \) and \( G \) is a chordal graph. 

\begin{theorem} \cite[Theorem 3.5]{RSK21}
    Let $G \in \GG$ be a chordal graph. Then, $\reg(S/J_G) \le \mathcal{C}(G).$
\end{theorem}

Additionally, Kumar \cite{AR3} also proved the conjecture for \( m = 2 \) in two distinct families of non-chordal graphs.

\begin{theorem} \cite[Theorems 3.7, 3.11 \& 3.15]{AR3}
    Let \( G \in \mathcal{G} \) be either a chordal graph, a quasi-block graph, or a semi-block graph. Then, $\reg(S/J_G) \le \mathcal{C}(G).$ Moreover, if $G$ has a cut edge, then $\reg(S/J_G) < \mathcal{C}(G).$
\end{theorem}

The work of \cite{RSK21, AR3} utilizes the regularity lemma on a short exact sequence derived from a simple but powerful lemma by Ohtani \cite{Oh11}. It also employs an inductive argument based on the number of internal vertices. Below, we present the above-mentioned lemma due to Ohtani:

\begin{lemma} \label{oh-lem}\cite[Lemma 4.8]{Oh11}
    Let $G\in \GG$. If $v$ is an internal vertex of $G$, then $$ J_G= J_{G_v} \cap ((x_v,y_v)+J_{G\setminus v}).$$ 
\end{lemma}

Kumar extended this result to the case of generalized binomial edge ideals:

\begin{theorem} \label{ak-thm}\cite[Theorem 3.2]{Kumar-reg}
    Let $G\in \GG$ and $m \ge 2$. If $v$ is an internal vertex of $G$, then $$ J_{K_m,G}= J_{K_m,G_v} \cap ((x_{i,v}~:~ 1 \le i \le m)+J_{K_m,G\setminus v}).$$ 
\end{theorem}

For each $v \in V(G)$, we set $P_v = (x_{i,v}~:~ 1 \le i \le m) \subseteq S$. \Cref{ak-thm} naturally provide the following short exact sequence: 

\begin{align}\label{ses}
    0 \to \frac{S}{J_{K_m,G}}\to \frac{S}{J_{K_m,G_v}} \oplus \frac{S}{P_v+J_{K_m,G\setminus v}} \to \frac{S}{P_v+J_{K_m,G_v \setminus v}}\to 0.
\end{align} 

In combinatorial approaches, one of the main proof techniques is induction on one of the parameters. In \cite{Kumar-reg}, Kumar proved the following lemma which turns out to be extremely useful in proofs using induction on the number of internal vertices. 

\begin{lemma}\label{iv-lem} \cite[Lemma 3.4]{Kumar-reg}
    Let $G\in \GG$. If $v$ is an internal vertex of $G$, then $\iv(G_v\setminus v)=\iv(G_v)<\iv(G)$ and $\iv(G\setminus v)<\iv(G).$
\end{lemma}

Using \Cref{ak-thm} and \Cref{iv-lem}, Kumar introduced an innovative inductive technique that allowed for significant progress on \Cref{reg-que} for any $m$: 

\begin{theorem} \label{up-vert} \cite[Theorems 3.6 \& 3.7]{Kumar-reg}
Let $G \in \GG$. Then, for any $m\ge 2$, $$\reg \left(S/J_{K_m,G}\right) \le n(G)-c(G).$$\footnote{The original statements by Kumar \cite[Theorem 3.6]{Kumar-reg}, was given for connected graphs. One can derive the present version as consequence of \Cref{tensor-res}.}  Moreover, if $m$ is at least the maximum number of vertices in a component of $G$, then $$\reg \left(S/J_{K_m,G}\right) = n(G)-c(G).$$ 
\end{theorem}

\Cref{up-vert} almost answers affirmatively \Cref{reg-que}, since for $m \ge \sqrt{\frac{2n(G)-2c(G)}{\mathcal{C}(G)}+\frac{1}{4}} +\frac{1}{2}$, we have  $n(G)-c(G) \le \min \left\{ { m \choose 2} \mathcal{C}(G), ~e(G)\right\}.$  In the case of chordal graphs, Kumar in \cite{Kumar-reg} not only resolved the \Cref{reg-que} but also obtained a better upper bound for the regularity of generalized binomial edge ideals of these graphs:

\begin{theorem}\label{reg-up-chordal-m} \cite[Theorem 3.13]{Kumar-reg}
    Let $G \in \GG$ be a chordal graph. Then, $$\reg \left(S/J_{K_m,G}\right) \le \min\{(m-1)\mathcal{C}(G), n(G)-c(G)\}.$$
\end{theorem}

Kumar proposed a conjecture \cite[Conjecture 3.15]{Kumar-reg} that enhances the upper bound proposed in \Cref{reg-que}. The conjecture is stated as follows: 

\begin{conjecture}\label{reg-up-conj-m}\cite[Conjecture 3.15]{Kumar-reg}
   Let $G \in \GG$. Then, $$\reg \left(S/J_{K_m,G}\right) \le \min\{(m-1)\mathcal{C}(G), n(G)-c(G)\}.$$ 
\end{conjecture}

In 2021, Rouzbahani Malayeri, Saeedi Madani, and Kiani proved the conjecture specifically for the case when \( m = 2 \). Their work not only validated the conjecture but also established an improved version of \Cref{reg-que}. They introduced a new combinatorial invariant that, to our knowledge, is not currently found in the graph theory literature. This new invariant is defined as follows:

A set of edges \( \mathcal{E} \subseteq E(G) \) is called a {\it clique-disjoint set} of a graph \( G \) if no two edges in \( \mathcal{E} \) belong to a clique of $G$. Define
\[
\eta(G) := \max\{ |\mathcal{E}|~:~ \mathcal{E} \text{ is a clique-disjoint set of } G\}.
\] 

\begin{figure}[H]
\begin{tikzpicture}
\draw (2,2)-- (0,0);
\draw (2,2)-- (4,0);
\draw (0,0)-- (4,0);
\draw (1,1)-- (3,1);
\draw (3,1)-- (2,0);
\draw (2,0)-- (1,1);
\begin{scriptsize}
\fill [color=black] (2,2) circle (1.5pt);
\fill [color=black] (0,0) circle (1.5pt);
\fill [color=black] (4,0) circle (1.5pt);
\fill [color=black] (1,1) circle (1.5pt);
\fill [color=black] (3,1) circle (1.5pt);
\fill [color=black] (2,0) circle (1.5pt);
\end{scriptsize}
\end{tikzpicture}    
\caption{Graph $H$}
\label{fig:eta-clique-no}
\end{figure}

While it may look like the set $\mathcal{E}$ consists of an edge each from a maximal clique and hence the invariant $\eta(G)$ being equal to $\mathcal{C}(G)$, the invariant $\eta(-)$ can be strictly smaller than $\mathcal{C}(-)$. For example, in the graph $H$ given in \Cref{fig:eta-clique-no}, $\eta(H) =  3 < \mathcal{C}(G) = 4$.

\begin{theorem}\label{reg-up-RSK-1}\cite[Corollary 2.7]{RSK2}
Let $G \in \GG$. Then, $\reg \left(S/J_{G}\right) \le \eta(G).$
\end{theorem}

Since for any \( G \in \mathcal{G} \), \( \eta(G) \leq \mathcal{C}(G) \), this result resolves \Cref{reg-que} in the case when \( m = 2 \). The approach taken by Rouzbahani Malayeri, Saeedi Madani, and Kiani to prove \Cref{reg-up-RSK-1} is particularly intriguing. They utilized an inductive method grounded in \Cref{oh-lem} and \Cref{iv-lem}, introducing the concept of a compatible map. Here, we revisit the concept of a compatible map as discussed in \cite{RSK2}.

For a graph \( G \), let \( \widehat{G} \) represent the induced subgraph obtained by removing all isolated vertices from \( G \). If \( G \) contains no isolated vertex, then \( \widehat{G} = G \).

\begin{definition}\label{def-com-RSM2} \cite[Definition 2.1]{RSK2}
A map $\psi: \mathcal{G} \to \mathbb{N}$ is said to be compatible if it satisfies the following conditions: 
\begin{enumerate}
\item[(a)] $\psi(\widehat{G}) \le \psi(G)$ for all $G\in \GG$;
\item[(b)] if $G= \sqcup_{i=1}^{c(G)} K_{n_i}$ with $n_i \ge 2$ for every $1 \le i \le c(G)$, then $\psi(G) \ge c(G)$;
\item[(c)]  for a graph $G \in \GG$, if $G \neq  \sqcup_{i=1}^{c(G)} K_{n_i}$, then there exists a vertex $v \in V(G)$ so that  $$ \psi(G\setminus v) \le \psi (G) \text{ and } \psi(G_v) < \psi(G).$$  
\end{enumerate}
\end{definition}

Building on the indicative techniques from \cite{Kumar-reg, AR3}, Rouzbahani Malayeri, Saeedi Madani, and Kiani successfully established the following result: 

\begin{theorem}\label{reg-up-RSK} \cite[Theorem 2.4]{RSK2}
    Let $\psi~:~ \GG \to \NN$ be a compatible map. Then, $\reg \left(S/J_{G}\right) \le \psi(G).$
\end{theorem}

As Kumar noted in \cite{Kumar-reg}, the regularity of generalized binomial edge ideals depends on the choice of \( m \). For instance, as demonstrated in \Cref{up-vert}, if \( m \) is at least as large as the maximum number of vertices in any component of \( G \), then \( \reg(S/J_{K_m,G}) = n(G) - c(G) \). Additionally, as discussed in \cite[Proposition 3.3]{Kumar-reg}, \( \reg(S/J_{K_m,K_n}) = \min\{m-1,n-1\} \) and \cite[Theorem 3.12]{Kumar-reg}, \( \reg(S/J_{K_m,G}) = (m-1)\mathcal{C}(G) \) for some block graph $G$, therefore, it is important that the compatible map is dependent on \( m \). 

To extend the concept of a compatible map to the case where one can apply it in a more general situation, we propose the following notion of \((\reg,m)\)-compatible map:

\begin{definition}\label{def-com}
Let $m \ge 2$ be a positive integer. A map $\psi: \mathcal{G} \to \mathbb{N}$ is said to be $(\reg,m)$-compatible if it satisfies the following conditions: 
\begin{enumerate}
\item[(a)] $\psi(\widehat{G}) \le \psi(G)$ for all $G\in \GG$;
\item[(b)] if $G= \sqcup_{i=1}^{c(G)} K_{n_i}$ with $n_i \ge 2$ for every $1 \le i \le c(G)$, then $$\psi(G) \ge \sum_{i=1}^{c(G)}\min\{m-1, n_i-1\};$$
\item[(c)]  for a graph $G \in \GG$, if $G \neq  \sqcup_{i=1}^{c(G)} K_{n_i}$, then there exists an internal vertex $v \in V(G)$ so that either
$$ \psi(G\setminus v) \le \psi (G) \text{ and } \psi(G_v) < \psi(G),$$ or  
$$ \psi(G\setminus v) \le \psi (G), ~ \psi(G_v) = \psi (G) \text{ and } \psi(G_v \setminus v) < \psi(G).$$ 
\end{enumerate}
\end{definition}

We begin by demonstrating the existence of $(\reg,m)$-compatible maps. It is important to note that we have refined part (c) of \Cref{def-com-RSM2}. Furthermore, according to \Cref{def-com-RSM2}, the following map does not qualify as a compatible map. Therefore, $(\reg,m)$-compatible maps not only generalize \Cref{def-com} but also encompass many more maps beyond just compatible ones. For a graph $G$, let $\text{is}(G)$ denote the number of isolated vertices of $G$.

\begin{proposition}\label{comp-MM}
Let $\psi~:~\GG \to \NN$ be defined by $\psi(G)=n(G)-c(G)$.
Then, $\psi$ is a $(\reg,m)$-compatible map.
\end{proposition}
\begin{proof}
    Let \( G \in \mathcal{G} \) be any graph. We have 
\[
\psi(\widehat{G}) = n(\widehat{G}) - c(\widehat{G}) = n(\widehat{G}) + \text{is}(G) - \text{is}(G) - c(\widehat{G}) = n(G) - c(G) = \psi(G).
\]

Next, consider the case where \( G = \bigsqcup_{i=1}^{c(G)} K_{n_i} \) with \( n_i \geq 2 \) for every \( 1 \leq i \leq c(G) \). In this scenario, we have
\[
\psi(G) = n(G) - c(G) = \sum_{i=1}^{c(G)} (n_i - 1) \geq \sum_{i=1}^{c(G)} \min\{m - 1, n_i - 1\}.
\] 

Next, consider the scenario where \( G \neq \sqcup_{i=1}^{c(G)} K_{n_i} \). In this case, \( G \) has an internal vertex, denoted as \( v \). Since \( v \) is part of at least two maximal cliques, we get \( c(G \setminus v) \geq c(G) \). Hence, $\psi(G\setminus v) =n(G)-1-c(G\setminus v) < n(G)-c(G)=\psi(G). $ 

Also, note that $c(G_v)=c(G)=c(G_v \setminus v).$ Therefore, $\psi(G_v)=n(G_v)-c(G_v)=\psi(G)$ (This  shows that $\psi$ is not a compatible map) and $\psi(G_v \setminus v)=n(G_v)-1-c(G_v)=\psi(G)-1< \psi(G). $  Hence, $\psi$ is a $(\reg,m)$-compatible map. 
\end{proof}

We provide more examples of \((\reg,m)\)-compatible maps later in this section. Now, we establish the known bounds in the literature regarding the regularity of binomial edge ideals, specifically for the case when \(m=2\) and for general values of \(m\) by utilizing the concept of \((\reg,m)\)-compatible maps. We present one of the main new results of this section, which not only generalizes \Cref{reg-up-RSK} but also provides a concise and alternative proof for many regularity upper bound results found in the literature. 

\begin{theorem}\label{ubT}
Let $\psi~:~ \GG \to \NN$ be a $(\reg,m)$-compatible map. Then, for every $G \in \GG,$ $$\reg(S/J_{K_m,G}) \le \psi(G).$$
\end{theorem}
\begin{proof}
We prove the assertion by induction on $\iv(G)$, the number of internal vertices of $G$. 
If \(\iv(G) = 0\), then \(G\) is the disjoint union of complete graphs. Thus, we can express \(\widehat{G}\) as \(\widehat{G} = \sqcup_{i=1}^{c(\widehat{G})} K_{n_i}\), where \(n_i \geq 2\) for \(1 \leq i \leq c(\widehat{G})\). Since \(J_{K_m, G} = J_{K_m, \widehat{G}}\), it follows that \(\reg(S/J_{K_m, G}) = \reg(S/J_{K_m, \widehat{G}}) = \reg(\widehat{S}/J_{K_m, \widehat{G}})\), where \(\widehat{S} = K[x_{i,j} : i \in [m] \text{ and } j \in V(\widehat{G})]\).

From \cite[Proposition 3.3]{Kumar-reg}, we have: \[
\reg(S/J_{K_m, G}) = \sum_{i=1}^k \min\{m-1, n_i-1\}.
\]
Now, combining Parts (a) and (b) of \Cref{def-com}, 
\[
\reg(S/J_{K_m, G}) = \sum_{i=1}^k \min\{m-1, n_i-1\} \leq \psi(\widehat{G}) \leq \psi(G).
\]
Thus, the result holds true for graphs with \(\iv(G) = 0\).

Next, we assume that \(\iv(G) > 0\). According to Part (c) of \Cref{def-com}, there exists an internal vertex \(v \in V(G)\) such that \(\psi(G \setminus v) \leq \psi(G)\) and \(\psi(G_v) < \psi(G)\) or \(\psi(G \setminus v) \leq \psi(G)\), \(\psi(G_v) = \psi(G)\) and \(\psi(G_v \setminus v) < \psi(G)\).  By \Cref{iv-lem}, we have $\iv(G\setminus v) < \iv(G)$, $\iv(G_v) <\iv(G)$ and $\iv(G_v\setminus v) <\iv(G)$. Set $S_v=K[x_{i,j}~:~i \in [m] \text{ and } j \in V(G\setminus v)].$ Thus, by induction, $\reg(S_v/J_{K_m,G\setminus v}) \le \psi(G\setminus v)$, $\reg(S/J_{K_m,G_v}) \le \psi(G_v)$  and $\reg(S_v/J_{K_m,G_v\setminus v}) \le \psi(G_v\setminus v)$. 

Now,  consider the short exact sequence \eqref{ses}. 
Note that $$\reg\left(\frac{S}{P_v+J_{K_m,G\setminus v}}\right)=\reg\left(\frac{S_v}{J_{K_m,G\setminus v}}\right)\le \psi(G\setminus v)$$ and $$\reg\left(\frac{S}{P_v+J_{K_m,G_v\setminus v}}\right)=\reg\left(\frac{S_v}{J_{K_m,G_v\setminus v}}\right)\le \psi(G_v\setminus v).$$ 

By applying \Cref{reg-lemma} to the short exact sequence \eqref{ses}, we get \begin{align*} \reg(S/J_G) &\le \max\{\reg(S_v/J_{K_m,G\setminus v}), \reg(S/J_{K_m,G_v}), \reg(S_v/J_{K_m,G_v\setminus v}) +1 \} \\ & \le \max\{ \psi(G\setminus v), \psi(G_v), \psi(G_v\setminus v)+1 \} \le \psi(G), \end{align*} where the last inequality follows from the part $(c)$ of \Cref{def-com}.  Hence, for every $G\in \GG,$ $\reg(S/J_{K_m,G}) \le \psi(G).$
\end{proof}

As an immediate consequence of \Cref{comp-MM} and \Cref{ubT}, we not only reaffirm the regularity upper bound result established by Matsuda and Murai (\Cref{first-gen-up}) for the case when \(m=2\), but we also provide a concise and alternative proof of this result. Furthermore, we retrieve the upper bound for the general case that was established by Kumar, given in \Cref{up-vert}.

\begin{remark}
We would like to note that the bound in \Cref{up-vert} is indeed tight. If \( G \) is a path graph, then without any restrictions on \( m \ge 2 \), we have \( \reg(S/J_{K_m,G}) = n(G) - 1 \). Furthermore, as pointed out in \Cref{up-vert}, if \( m \geq n \), then for every graph \( G \in \GG \), it holds that \( \reg(S/J_{K_m,G}) = n(G) - c(G) \). This evidence suggests that \Cref{up-vert} represents a best possible bound. 
\end{remark}

For \( m = 2 \), the equality \( \reg(S/J_{G}) = n(G) - c(G) \) holds if and only if \( G \) is a disjoint union of paths. This was conjectured by Matsuda and Murai \cite{MM} and has been proven by Kiani and Saeedi Madani in \cite{KMJCTA}. 

Given the above remark, it is natural to pose the following question.

\begin{question}\label{ques-gen-mm}
For fixed \( 2 \le m \le n - 1 \), for which graphs does \( \reg(S/J_{K_m,G}) = n(G) - c(G) \) hold?
\end{question}

Next, we provide more examples of $(\reg,m)$-compatible maps. 

\begin{proposition}\label{comp-eta}
Let $\psi~:~\GG \to \NN$ be defined by $\psi(G)=(m-1) \eta(G)$. Then $\psi$ is a $(\reg,m)$-compatible map.
\end{proposition}
\begin{proof}
Let \( G \in \mathcal{G} \) be any graph. Since $E(\widehat{G})=E(G)$,  $\eta(\widehat{G})=\eta(G)$, and hence, 
\[
\psi(\widehat{G}) = (m-1)\eta(\widehat{G})=(m-1)\eta(G)=\psi(G).
\]

Now, let \( G = \bigsqcup_{i=1}^{c(G)} K_{n_i} \) with \( n_i \geq 2 \) for every \( 1 \leq i \leq c(G) \). In this scenario, $\eta(G)=\sum_{i=1}^{c(G)} \eta(K_{n_i})=c(G)$ as a complete graph itself is a  clique.  Thus, we have
\[
\psi(G) = (m-1)\eta(G) = (m-1)c(G) \geq  \sum_{i=1}^{c(G)} \min\{m - 1, n_i - 1\}.
\] 

Next, assume that \( G \neq \sqcup_{i=1}^{c(G)} K_{n_i} \). In this case, \( G \) has an internal vertex, say  \( v \). It follows from the proof of \cite[Theorem 2.6]{RSK2} that $\eta(G\setminus v) \le \eta(G)$ and $\eta(G_v) <\eta(G)$. Consequently,  $\psi(G\setminus v) =(m-1)\eta(G\setminus v) \le (m-1)\eta(G)=\psi(G)$, and $\psi(G_v )=(m-1)\eta(G_v) <(m-1)\eta(G)=\psi(G). $  Hence, $\psi$ is a $(\reg,m)$-compatible map.
\end{proof}

As an immediate consequence of \Cref{comp-eta} and \Cref{ubT}, we not only recover the regularity upper bound result by Rouzbahani Malayeri, Saeedi Madani, and Kiani (\Cref{reg-up-RSK-1}) for the case when \( m=2 \), but we also recover the regularity upper bound (\Cref{reg-up-chordal-m}) established by Kumar for chordal graphs for any \( m \ge 2 \). Additionally, we provide a positive answer to \Cref{reg-up-conj-m}: 
\begin{corollary}\label{cor-eta-m}
Let $G\in \GG$. Then, for any $m\ge 2$,
$$\reg(S/J_{K_m,G}) \le (m-1)\eta(G).$$
\end{corollary}
As we know, \(\eta(G) \leq \mathcal{C}(G)\). The combination of \Cref{cor-eta-m} and \Cref{up-vert} provides an affirmative answer to \Cref{reg-que} proposed by Saeedi Madani and Kiani. 

\begin{remark}
It can be noted that if \( G \) is a block graph, then \( \eta(G) = \mathcal{C}(G) \). As indicated in \cite[Theorem 3.12]{Kumar-reg}, if \( G \) is a block graph such that for every maximal clique \( \mathcal{V} \), the condition \( |\mathcal{V}| \geq m + \iv(\mathcal{V}) \) holds, where \( \iv(\mathcal{V}) \) represents the number of internal vertices of \( G \) that belong to \( \mathcal{V} \), then we have \( \reg(S/J_{K_m,G}) = (m-1)\eta(G) \). This result demonstrates that the bound presented in \Cref{cor-eta-m} is indeed attainable. 
\end{remark}
Given the previous remark, it is reasonable to pose the following question: 
\begin{question}
Characterize completely the graphs for which $\reg(S/J_{K_m,G})=(m-1)\eta(G).$
\end{question}

We will now discuss another recent upper bound for the regularity of binomial edge ideals, established by Ene, Rinaldo, and Terai in 2020 \cite{ERT20}. This upper bound improves upon the result presented in \Cref{first-gen-up}. 

Recall that the {\it clique number} of a connected graph \(G\), denoted as \(\omega(G)\), refers to the maximum size of a clique within \(G\). We introduce a new invariant defined as follows: suppose \(G = \sqcup_{i=1}^{c(G)} G_i\), 
\[
\gamma_m(G) := \sum_{i=1}^{c(G)} \min\{m-\omega(G_i)-1,-1\}.
\]

To illustrate $\gamma_m(G)$, we use the graph in \Cref{example:gen-graph}. For  the graph $G$ in \Cref{example:gen-graph}, we can observe that $\omega(G)=4$ and that $G$ consists of only one component. Therefore, $\gamma_2(G)=-3$, $\gamma_3(G)=-2$, and $\gamma_m(G)=-1$ for all $m \ge 4$.

Note that $\gamma_2(G)=\sum_{i=1}^{c(G)} \min\{1-\omega(G_i),-1\}=-\sum_{i=1}^{c(G)} \max\{\omega(G_i)-1,1\}.$ So, if $E(G_i)\neq \emptyset$ and  $\Delta(G_i)$ is the clique complex of $G_i$, then $\dim(\Delta(G_i))=\max\{\omega(G_i)-1,1\}.$ And if $E(G_i) =\emptyset$, then $\dim(\Delta(G_i))=0$. Thus, $n(G)+\gamma_2(G) \le n(G) -\sum_{i=1}^{c(G)} \dim(\Delta(G_i)).$ Ene, Rinaldo and Terai \cite{ERT20} proved the following: 

\begin{theorem}\label{ert-reg-up}\cite[Corollary 2.2]{ERT20}
Let $G\in \GG.$ Then, $\reg(S/J_G) \le n(G)- \sum_{i=1}^{c(G)} \dim(\Delta(G_i)).$
\end{theorem}

Notice that \Cref{ert-reg-up} improves the upper bound for the regularity of binomial edge ideals that was previously established by Matsuda and Murai (\Cref{first-gen-up}). This enhancement applies specifically to graphs \( G \) that do not have isolated vertices and contain at least one triangle. Our goal is to extend this result for a general \( m \) by utilizing the new invariant \( \gamma_m(G) \). The concept of this new invariant was inspired by the work of Kumar, which states the following: 

\begin{theorem}\label{reg-b-ub}\cite[Theorem 3.10]{Kumar-reg}
    Let $G \in \GG$ be a connected graph. If $2 \le m \le \omega(G)-1$, then $\reg(S/J_{K_m,G}) \le n(G)-2. $
\end{theorem}

We now extend \Cref{ert-reg-up} for a general value of \( m \) and strengthen \Cref{reg-b-ub}. Additionally, we unify these two results using the concept of a \((\reg, m)\)-compatible map. It is important to note that the five-page proof provided by Ene, Rinaldo, and Terai \cite[Theorem 2.1]{ERT20} employs a completely different approach, while our proof can be categorized under the framework of \((\reg, m)\)-compatible maps.  

\begin{proposition}\label{prop-comp-gamma}
Let $\psi~:~\GG \to \NN$ be defined by $\psi(G)=n(G)+\gamma_m(G)$. Then $\psi$ is a $(\reg, m)$-compatible map.
\end{proposition}
\begin{proof}
     Let \( G \in \mathcal{G} \) be any graph. Since $E(\widehat{G})=E(G)$, we get \begin{align*} \gamma_m(G)&=\sum_{i=1}^{c(G)} \min\{m-\omega(G_i)-1,-1\} =\gamma_m(\widehat{G})-\text{is}(G).\end{align*} Therefore,  
\begin{align*}
    \psi(\widehat{G}) &= n(\widehat{G})+\gamma_m(\widehat{G})=(n(G)-\text{is}(G))+(\gamma_m(G)+\text{is}(G))=\psi(G).
\end{align*}

Next, consider the case when  \( G = \sqcup_{i=1}^{c(G)} K_{n_i} \) with $n_i \ge 2$ for every $1 \le i \le c(G).$ In this situation, since each component is a complete graph, \( \gamma_m(G) = \sum_{i=1}^{c(G)} \min\{m-n_i-1,-1\} \). Therefore,  
\[
\psi(G) = n(G)+\gamma_m(G)= \sum_{i=1}^{c(G)} n_i +\sum_{i=1}^{c(G)} \min\{m-n_i-1,-1\} =\sum_{i=1}^{c(G)} \min\{m-1,n_i-1\} .
\] 

Now, assume that \( G \neq \sqcup_{i=1}^{c(G)} K_{n_i} \). In this case, \( G \) has an internal vertex, say \( v \). Suppose $G=\sqcup_{i=1}^{c(G)} G_i$. Without loss of generality, we may assume that $v \in V(G_1).$ We claim that $\omega(({G_1})_v) \ge \omega(({G_1})_v \setminus v) \ge \omega(G_1).$ Let $\mathcal{V}$ be a maximal clique of $G_1$ with size $\omega(G_1).$ If $v \in \mathcal{V}$, then the maximal clique of $(G_1)_v$ containing $v$ also contains $N_{G_1}(v) \cup \{ v\}$. Since $v$ is an internal vertex of $G_1$, we get that $\mathcal{V}$ is a proper subset of $ N_{G_1}(v) \cup \{v\}.$ Therefore, $\omega((G_1)_v) \geq |N_{G_1}(v) \cup \{v\}| > |\mathcal{V}|=\omega(G_1).$ Consequently, $\omega((G_1)_v \setminus v) \ge \omega((G_1)_v)-1 \ge \omega(G_1).$ Suppose $v \not\in \mathcal{V}.$ Then, $\mathcal{V}$ is still a clique in $(G_1)_v$ and $(G_1)_v \setminus v.$ Therefore, $\omega((G_1)_v) \ge \omega((G_1)_v\setminus v) \ge  |\mathcal{V}| =\omega(G_1).$  

Thus, using the fact that $n(G_v)=n(G)$ and the formula for the invariant $\gamma_m(-)$ , we get \begin{align*}
    \psi(G_v)&=n(G_v)+\gamma_m(G_v)\\&=n(G)+\min\{m-\omega((G_1)_v)-1,-1\}+\sum_{i=2}^{c(G)} \min\{m-\omega(G_i)-1,-1\}\\&\le n(G)+\min\{m-\omega(G_1)-1,-1\}+\sum_{i=2}^{c(G)} \min\{m-\omega(G_i)-1,-1\}\\&=n(G)+\gamma_m(G)=\psi(G).
\end{align*}

Similarly, we obtain
\begin{align*}
    \psi(G_v\setminus v)&=n(G_v\setminus v)+\gamma_m(G_v\setminus v)\\&=n(G)-1+\min\{m-\omega((G_1)_v\setminus v)-1,-1\}+\sum_{i=2}^{c(G)} \min\{m-\omega(G_i)-1,-1\}\\&\le n(G)-1+\min\{m-\omega(G_1)-1,-1\}+\sum_{i=2}^{c(G)} \min\{m-\omega(G_i)-1,-1\}\\&=n(G)+\gamma_m(G)-1<\psi(G).
\end{align*}

It remains to prove that $\psi(G\setminus v) \le \psi(G).$ Let $H_{1,1},\ldots,H_{1,k}$ be components of $G_1 \setminus v$. This implies that $G\setminus v= \sqcup_{j=1}^k H_{1,j} \sqcup \sqcup_{i=2}^{c(G) } G_i.$ Note that there exists   $1 \le r \le k$ such that $\omega(H_{1,r}) = \omega(G_1\setminus v) \ge \omega(G_1)-1.$ Therefore, \begin{align*}
    \psi(G\setminus v)&=n(G\setminus v)+\gamma_m(G\setminus v)\\
    &= n(G)-1+ \sum_{j=1}^k \min\{m-\omega(H_{1,j})-1,-1\} + \sum_{i=2}^{c(G)}\min\{m-\omega(G_i)-1,-1\}
    \\
    &\le n(G)-1+  \min\{m-\omega(H_{1,r})-1,-1\} + \sum_{i=2}^{c(G)}\min\{m-\omega(G_i)-1,-1\} \\ 
    &\le n(G) + \min\{m-\omega(H_{1,r})-2,-2\} + \sum_{i=2}^{c(G)}\min\{m-\omega(G_i)-1,-1\} \\ 
    &\le n(G)+\sum_{i=1}^{c(G)}\min\{m-\omega(G_i)-1,-1\}=\psi(G),
\end{align*}
where the last inequality holds true since $\omega(H_{1,r}) \geq \omega(G_1) - 1$. Hence, $\psi$ is a $(\reg,m)$-compatible map.
\end{proof}

As an immediate consequence of \Cref{ubT} and \Cref{prop-comp-gamma}, we obtain a new upper bound on the regularity of generalized binomial edge ideals.  
\begin{corollary}\label{reg-up-gamma}
    Let $G \in \GG. $ Then, $\reg(S/J_{K_m,G}) \le n(G)+\gamma_m(G).$
\end{corollary}

We have seen three different $(\reg,m)$-compatible maps. It is natural to ask whether there are more $(\reg,m)$-compatible maps. We end this section by asking the following question.

\begin{question}
Do more $(\reg,m)$-compatible maps exist? 
\end{question}

\section{Upper bounds for specific classes of graphs}\label{sec:ub-specific}
In the previous section, we discussed general upper bounds for the regularity of (generalized) binomial edge ideals associated with arbitrary graphs. While these bounds are tight for certain classes of graphs, they often fail to be optimal for many other classes. This observation naturally leads to the pursuit of sharper upper bounds or precise computations of regularity for well-studied classes of graphs in graph theory. In this section, we review the advancements in the literature aimed at improving the regularity upper bounds for specific graph classes.

The initial breakthroughs in this area were contributed by Jayanthan, Narayanan and Raghavendra Rao. Their work focused on trees, particularly lobster trees, and introduced tighter bounds based on the structural parameters of the graph. A \textit{spine} of a tree $G$ is a longest path in $G$. For a tree $G$, let $\mathcal{L}(G)$ denote the set of all vertices of degree one in $G$. A tree $G$ is said to be a \textit{caterpillar} if the induced subgraph on $V(G) \setminus \mathcal{L}(G)$ is either an empty graph or a path. A tree $G$ is said to be a \textit{lobster} if the induced subgraph on $V(G) \setminus \mathcal{L}(G)$ is a caterpillar.

\begin{theorem}\cite[Corollary 4.8]{JNR}
If $G$ is a lobster tree with a spine $P$ of length $\ell$ with $t$ limbs and $r$ whiskers attached to the spine $P$, then $\reg(S/J_G) \leq \ell + 2t.$ \end{theorem}

Here, a \textit{limb} is a $K_{1,r}$ and a \textit{whisker} is a an edge. Further refinements were made for general trees by incorporating additional graph parameters:

\begin{theorem}\cite[Corollary 2]{JNR2}
Let $T$ be a tree on $[n]$ with a spine P of length $\ell$. Let $e_2$ denote the
number of edges that are not in $P$ and with both endpoints having a degree at most $2$
and $d_3$ denote the number of vertices, not in $P$, and having degree at least $3$. Then
$\reg(S/J_T ) \leq e_2 + \ell + 2d_3$.
\end{theorem}

Kumar extended these ideas to generalized block graphs by introducing bounds that depend on the interplay between pendent vertices and cliques in the graph. Before stating the result, we briefly recall some key invariants from graph theory:

A vertex $v$ is said to be a \textit{pendent} vertex, if $\deg_G(v) =1$. For $v \in V(G)$, let $\text{cdeg}_G(v)$ denote the number of maximal cliques that contain $v$, and $\text{pdeg}_G(v)$ denote the number of pendent vertices adjacent to $v$. Note that for every $v \in V(G)$, $\text{pdeg}_G(v) \leq \text{cdeg}_G(v)$. A vertex $v \in V(G)$ with $\text{pdeg}_G(v) \geq 1$ is said to be of type $1$, if $\text{cdeg}_G(v) = \text{pdeg}_G(v)+1$ and of type $2$ if $ \text{cdeg}_G(v) \geq \text{pdeg}_G(v) +2$. We denote by $\alpha(G)$ the number of vertices of type $1$ in $G$ and by $\text{pv}(G)$ the number of pendent vertices.
            
\begin{theorem}\label{K-gbg}\cite[Section 4]{AR2}
Let $G$ be a generalized block graph. Then, $\reg(S/J_G) \le \mathcal{C}(G)+\alpha(G)-\text{pv}(G).$ 
\end{theorem}

A graph $G$ is said to be \textit{Cactus graph} if each block of $G$ either a cycle or an edge. Recently Jayanthan and Sarkar made efforts to improve the regularity of binomial edge ideals of Cactus graphs, resulting in an improved upper bound. They, in fact, considered a slightly more general class of graphs and derived an optimal upper bound for this class.

\begin{theorem}\label{JS-cac}\cite[Theorem 3.2]{JS22}
Let $G$ be a graph such that each block of $G$ is either a cycle or a clique. Let $c'(G)$ denote the number of maximal cliques of $G$, except the edges of any cycle of length $k\geq 4$ in $G$, and let $c_k(G)$ denote the number of cycles of $G$ of length $k$. Then $\reg(S/J_G) \leq c'(G)+\sum_{k\geq 4}(k-2)c_k(G)$. 
\end{theorem}

\Cref{K-gbg} and \Cref{JS-cac} present the best-known upper bounds for the regularity of binomial edge ideals associated with generalized block graphs and cactus graphs, respectively, in the existing literature. It is natural to explore whether these findings can be applied to generalized binomial edge ideals. This leads to the following two questions:
\begin{question}
    What is the most effective invariant that provides the best possible upper bound on the regularity of the generalized binomial edge ideals of cactus graphs?
\end{question}
\begin{question}
    What is the most effective invariant that provides the best possible upper bound on the regularity of the generalized binomial edge ideals of generalized block graphs?
\end{question}

Now, we end this section by discussing the Matsuda and Murai conjecture. 

\begin{conjecture} \cite[Conjecture 3.10]{MM}
Let $G \in \GG$. Then, $\reg(S/J_G) = n(G) - 1$ if and only if $G$ is a path graph on $n(G)$ vertices. \end{conjecture}

As mentioned in the previous section, Matsuda and Murai proved the `if part'. Initial attempts to prove the `only if part' of the conjecture were made by Zafar and Zahid \cite{ZZ13}, who proved that cycle of length $n$ has regularity $n-2$, giving an evidence for this conjecture. In 2015, Ene and Zarojanu \cite{EZ} proved the conjecture for block graphs. Kiani and Saeedi Madani \cite{KMJCTA}, in 2015, established this conjecture affirmatively. We present here a simple proof of this result using the techniques we have developed in the previous section: 

\begin{theorem}\label{mm-conj}\cite[Theorem 3.2]{KMJCTA}
Let $G \in \GG$. If $G$ is not a path graph on $n(G)$ vertices, then $\reg(S/J_G) \le n(G)-2.$ 
\end{theorem}
\begin{proof}
It follows from \Cref{reg-up-gamma}  that $\reg(S/J_G) \le n(G)+\gamma_2(G).$ Suppose $\gamma_2(G) \le -2$, then $\reg(S/J_G) \le n(G)+\gamma_2(G) \le n(G)-2,$ and we are done. So, we assume that $\gamma_2(G)=-1$, i.e., $G$ is a connected graph and $\omega(G)=2$. We proceed by induction on $\iv(G)$. Suppose $\iv(G)=0$, then $G$ is an edge, a contradiction to the assumption that  $G$ is not a path. 

Assume now that $\iv(G) \ge 1$. Let $v$ be a vertex of maximal degree. Since $\omega(G)=2$ and $G$ is not a path graph, $v$ has to be an internal vertex. Then, due to \Cref{oh-lem}, we have the following short exact sequence: \[0 \to \frac{S}{J_G} \to \frac{S}{J_{G_v}} \oplus \frac{S}{(x_v,y_v)+J_{G\setminus v}} \to  \frac{S}{(x_v,y_v)+J_{G_v\setminus v}} \to 0.\]

Now, we have two cases, $\deg_G(v) \ge 3$ or $\deg_G(v)=2.$ Suppose $\deg_G(v)\ge 3$. Then, $N_G(v)\cup \{v\}$ is a clique in $G_v$, and therefore, $\omega(G_v)\ge |N_G(v)\cup\{v\}| = \deg_G(v)+1\ge 4.$ This implies that $\omega(G_v\setminus v) \ge 3.$ Thus, using \Cref{reg-up-gamma}, $\reg(S/J_{G_v}) \le n(G)+\gamma_2(G_v)= n(G)-\omega(G_v)+1 \le n(G)-3$ and $\reg(S/((x_v,y_v)+J_{G_v\setminus v})) \le n(G_v\setminus v)+\gamma_2(G_v\setminus v)=n(G)-1-\omega(G_v\setminus v)+1 \le n(G)-3.$  Since $G\setminus v$ is a graph on $n(G)-1$ vertices, by using \Cref{first-gen-up}, we get that $\reg(S/((x_v,y_v)+J_{G\setminus v})) \le n(G)-1-c(G\setminus v) \le n(G)-2.$ Thus, applying \Cref{reg-lemma}, we get that $\reg(S/J_G) \le n(G)-2.$

Next, suppose that \(\deg_G(v) = 2\). Since \(v\) is a vertex of maximal degree and \(G\) is not a path graph, we deduce that \(G\) must be a cycle graph. In this case, it is known from \cite{ZZ13} that \(\reg(S/J_G) = n(G) - 2\). Hence, in every scenario, we have \(\reg(S/J_G) \leq n(G) - 2\).
\end{proof}

It is natural to consider the analog of the Matsuda and Murai conjecture for generalized binomial edge ideals when \( m \geq 2 \). We raised this question in Section 3; refer to Question \ref{ques-gen-mm} for more details.

\section{Lower bounds and Equalities}\label{sec:lower-bound}

Obtaining a precise expression for the regularity of homogeneous ideals in terms of the data associated with the ideal is always a challenging problem. In a general setting, what one expects is a good bound, either lower or upper. By ``good'', we mean that this should be achieved by a large class of ideals. It is almost impossible to get a precise expression for the regularity of the binomial edge ideal of a generic graph in terms of the combinatorial data associated with the graph. So, researchers have been trying to obtain efficient upper or lower bounds by restricting their attention to special classes of graphs. In this section, we discuss some of the important lower bounds available in the literature along with some precise expressions for the regularity for some specific classes of graphs. 

Matsuda and Murai proved that if $H$ is an induced subgraph of $G$, then by restricting the minimal free resolution of $J_G$ to the variables associated with $V(H)$, one obtains a minimal free resolution of $J_H$. As an immediate consequence, they obtained $\beta_{ij}(S/I(H)) \leq \beta_{ij}(S/I(G))$ for all $i, j \geq 0$ which yielded a lower bound for the regularity:

\begin{corollary}\cite[Corollary 2.3]{MM} \label{lowerbound:length}
If $G$ is a graph, then $\ell(G) \leq \reg(S/J_G)$, where $\ell(G)$ denotes the length of a longest induced path in $G$.    
\end{corollary}

There are several classes of graphs for which this lower bound is achieved. Chaudhry, Dokuyuchu and Irfan studied the regularity of block graphs \cite{FC}. They proved:
\begin{theorem}\cite[Theorem  3.4]{FC}
Let $G$ be a connected graph on the vertex set $[n]$ which consists of
\begin{enumerate}
    \item[(i)] a sequence of maximal cliques $F_1,\dots, F_\ell$ with $|V(F_i)| \geq 2$ for all $i$ such that $|V(F_i)\cap V(F_{i+1})| = 1$ for $1 \leq i \leq \ell - 1$ and $V(F_i) \cap V(F_j) = \emptyset$ for any $i < j$ such that $j  \neq i + 1$, together with
    \item[(ii)] some additional edges of the form $F = \{j, k\}$ where $j$ is an intersection point of two consecutive cliques $F_i, F_{i+1}$ for some $1 \leq i \leq \ell - 1$, and $k$ is a vertex of degree $1$.
\end{enumerate}
Then $\reg(S/J_G) = \reg(S/\ini(J_G)) = \ell(G)$.
\end{theorem}

For trees, they characterized when equality occurs in the above inequality:
\begin{theorem}\cite[Theorem 4.1]{FC}
Let $T$ be a tree on $[n]$. Then $\reg(S/J_T) = \ell(T)$ if and only if $T$ is a caterpillar.
\end{theorem}

Ene and Zarojanu showed that for closed graphs, the lower bound of regularity is achieved.

\begin{theorem}\cite[Theorem 2.2]{EZ}
Let $G$ be a connected closed graph. Then $$\reg(S/J_G) = \reg(S/\ini(J_G)) = \ell(G).$$    
\end{theorem}

Zafar and Zahid proved that a cycle on $n$ vertices, $C_n$, has $\reg(S/J_{C_n}) = n-2$ by explicitly computing the Betti table, \cite{ZZ13}. Given \Cref{lowerbound:length} and \Cref{mm-conj}, this is now evident. Since $P_{n-1}$ is an induced subgraph of $C_n$, by \cite[Corollary 2.3]{MM}, we have $\reg(S/J_{C_n}) \geq n-2$. Moreover, since $C_n$ is not a path, by \Cref{mm-conj}, it follows that $\reg(S/J_{C_n}) \leq n-2$. 

Sarkar investigated the regularity of cycles that were accompanied by some whiskers. He proved the following:
\begin{theorem}\cite[Corollary 4.10]{Sar21}
If $G$ is a graph obtained by adding finitely many whiskers to the vertices, not necessarily all, of a cycle $C_k$, then $k-1 \leq \reg(S/J_G) \leq k+1$. Further, Let $A = \{v \in V(C_k) : \exists~ \text{a whisker at } v\}$. Then
\begin{enumerate}
    \item $\reg(S/J_G) = k + 1$ if and only if $A = V(C_k)$,
    \item $\reg(S/J_G) = k - 1$ if and only if $|A| = 1$ or $|A| = 2$ and vertices of A are adjacent,
    \item $\reg(S/J_G) = k$ if and only if $A$ contains at least two non-adjacent vertices and $A \subsetneq V(C_k).$
\end{enumerate}
\end{theorem}

While the length of the longest induced path serves as a sharp lower bound for the regularity of trees, it is evident that, for a significant class of trees, this bound is not particularly tight. Jayanthan, Narayanan, and Raghavendra Rao studied various classes of block graphs and proposed an improved lower bound for the regularity of trees. They also characterized all trees that achieve this lower bound.


\begin{theorem}\cite[Theorems  4.1 \& 4.2]{JNR}
For a tree $T$, $\reg(S/J_T) \geq \iv(T) + 1$. Moreover, the equality is achieved if and only if $T$ does not contain the Jewel graph (see \Cref{jewel}) as a subgraph.
\end{theorem}

\vskip 2mm \noindent
\textbf{Sketch of Proof:} Given a tree $T$, it has a leaf order on its vertices. Using this leaf ordering of the vertices and then inducting on the number of vertices, one can easily prove the lower bound. To prove the characterization for the equality, first it is easy to verify using Macaulay2 \cite{M2} that the regularity of $S/J_{\mathcal{J}} = 6 = iv(\mathcal{J}) + 2$. If $v_1,\ldots, v_n$ is a leaf order on the vertices of a tree $T$ having $\mathcal{J}$ as subgraph, then note that at each stage $G_i$, the induced graph on $v_1,\ldots, v_i$, addition of the next vertex $v_{i+1}$ will increase the number of internal vertices if and only if $v_i$ is a pendant vertex in $G_i$. In such case, the binomial corresponding to $v_{i+1}$ is a regular element on $J_{G_i}$. Therefore regularity and the number of internal vertices increase by one. If $v_i$ is not a pendant vertex in $G_i$, then $iv(G_{i+1}) = iv(G_i)$ and $\reg(S/J_{G_{i+1}}) \geq \reg(S/J_{G_i})$. Hence the result follows by induction.

For the converse, the first one shows that if $T$ does not contain a vertex of degree $2$, then it contains $\mathcal{J}$ as a subgraph. Secondly, one proves that if $T$ contains a vertex of degree $2$, then one can see $T$ as a clique sum of two trees $T_1$ and $T_2$ along a pendant vertex of each of these graphs. By \cite{JNR}, $\reg(S/J_T) = \reg(S/J_{T_1}) + \reg(S/J_{T_2})$ and $\iv(T) = \iv(T_1)+\iv(T_2) + 1$. Since $\reg(S/J_T) \geq iv(T)+2$, $\reg(S/J_{T_i}) \geq \iv(T_i)+2$ for at least one $i\in \{1,2\}$. Applying induction on the number of vertices, $T_i$ and hence $T$ has $\mathcal{J}$ as an induced subgraph.
\qed

\begin{figure}[H]
\begin{tikzpicture}[line cap=round,line join=round,>=triangle 45,x=0.6cm,y=0.6cm]
\draw (2.82,-2.52)-- (2.82,-3.86);
\draw (2.82,-2.52)-- (3.86,-2.02);
\draw (2.82,-2.52)-- (1.78,-2.04);
\draw (3.86,-2.02)-- (4.82,-2.52);
\draw (3.86,-2.02)-- (3.84,-1.1);
\draw (2.82,-3.86)-- (3.82,-4.52);
\draw (2.82,-3.86)-- (1.82,-4.52);
\draw (1.78,-2.04)-- (0.82,-2.52);
\draw (1.78,-2.04)-- (1.76,-1.02);
\begin{scriptsize}
\fill [color=black] (2.82,-2.52) circle (1.5pt);
\fill [color=black] (3.86,-2.02) circle (1.5pt);
\fill [color=black] (2.82,-3.86) circle (1.5pt);
\fill [color=black] (1.78,-2.04) circle (1.5pt);
\fill [color=black] (4.82,-2.52) circle (1.5pt);
\fill [color=black] (3.84,-1.1) circle (1.5pt);
\fill [color=black] (3.82,-4.52) circle (1.5pt);
\fill [color=black] (1.82,-4.52) circle (1.5pt);
\fill [color=black] (0.82,-2.52) circle (1.5pt);
\fill [color=black] (1.76,-1.02) circle (1.5pt);
\end{scriptsize}
\end{tikzpicture}
\caption{$\mathcal{J}$: Jewel}
\label{jewel}

\setlength{\unitlength}{0.4cm}
\begin{picture}(8,9)
\newsavebox{\Tri}

\savebox{\Tri}
  (04,03)[bl]{
  \put(00,00){\circle*{.3}}
  \put(04,00){\circle*{.3}}
  \put(02,03){\circle*{.3}}

  \put(00,00){\line(2,3){2}}
  \put(00,00){\line(1,0){4}}
  \put(02,03){\line(2,-3){2}}
}

\put(00,03){\usebox{\Tri}}
\put(04,03){\usebox{\Tri}}

\put(04,03){\line(-1,-1){2}}
\put(02,01){\line(-1,0){2}}
\put(02,01){\line(0,-1){2}}

\put(04,03){\line(1,-1){2}}
\put(06,01){\line(1,0){2}}
\put(06,01){\line(0,-1){2}}

\put(02,01){\circle*{.3}}
\put(00,01){\circle*{.3}}
\put(02,-01){\circle*{.3}}
\put(06,01){\circle*{.3}}
\put(08,01){\circle*{.3}}
\put(06,-01){\circle*{.3}}

\put(02.8,6.5){\circle*{.2}}
\put(03.6,6.8){\circle*{.2}}
\put(04.4,6.8){\circle*{.2}}
\put(05.2,6.5){\circle*{.2}}

\put(02.8,-01.5){\circle*{.2}}
\put(03.6,-01.8){\circle*{.2}}
\put(04.4,-01.8){\circle*{.2}}
\put(05.2,-01.5){\circle*{.2}}

\put(03.8,1.9){\boldmath{$v$}}
\end{picture}

\vspace{0.5cm}
\caption{A flower graph $F_{h,k}(v)$}\label{Big Flower}
\label{flower}
\end{figure}

Herzog and Rinaldo further generalized this lower bound to the case of block graphs.  A graph is said to be \textit{indecomposable} if it does not have any vertex of clique degree $2$. While investigating the regularity of binomial edge ideals of graphs, it is enough to consider indecomposable graphs due to \cite[Theorem 3.1]{JNR}.
\begin{theorem}\cite[Theorem 8]{HR-Extremal}
If $G$ is an indecomposable block graph, then $\reg(S/J_G) \geq \iv(G) + 1$.
\end{theorem}

Mascia and Rinaldo characterized all block graphs that attain equality in the above inequality.

\begin{theorem}\cite[Corollary 3.8]{RCJAA}
Let $G$ be a block graph. Then $\reg(S/J_G) = \iv(G)+1$ if and only if $G$ does not contain the flower graph (see \Cref{flower}) as an induced subgraph.
\end{theorem}

It is not very difficult to see that this lower bound does not extend to the class of bipartite graphs.  It can easily be seen that the number of internal vertices cannot be a lower bound since a bipartite graph may have all its vertices being internal vertices, for example, cycles of even length. So, the question is whether one has some fine lower bound or precise expression for the regularity of bipartite graphs. In general, there are no better lower bounds known than the length of the longest induced path. Schenzel and Zafar proved that for any complete bipartite graph $G$, $\reg(S/J_G) = 2$, which is the length of a longest induced path in $G$, \cite{Schenzel}.

Another homological invariant associated with a ring or a module is its \textit{depth}. One important  instance in the study of depth is the property of being \textit{Cohen-Macaulay}. An ideal  $I$ in a ring $R$ is said to be Cohen-Macaulay if $R/I$ is Cohen-Macaulay. We say that a graph $G$ is Cohen-Macaulay if $S/J_G$ is Cohen-Macaulay. Bolognini, Macchia and Strazzanti characterized the structure of Cohen-Macaulay bipartite graphs, \cite{BMS18}. In \cite{JA1}, we used this structure to computed the regularity of Cohen-Macaulay bipartite graphs. We describe it below.

Let $F_m$ denote the graph on the vertex set $[2m]$ with the edge set
$E(F_m) = \{\{2i, 2j-1\} : i = 1, \ldots , m, j = i, \ldots , m\}.$ They showed that these graphs are Cohen-Macaulay and are the building blocks of Cohen-Macaulay bipartite graphs. They proved that any Cohen-Macaulay bipartite graph is obtained by certain gluing operations on the $F_m$ graphs.

\textbf{Operation $*$:} For $i = 1, 2$, let $G_i$ be a graph with at least one vertex $f_i$ of degree one. Define the graph $(G_1, f_1) * (G_2, f_2)$ to be the graph  obtained by identifying $f_1$ and $f_2$.

\textbf{Operation $\circ$:} For $i = 1, 2$, let $G_i$ be a graph with at least one vertex $f_i$ of degree one and $v_i$ be its neighbor with $\deg_{G_i} (v_i) \geq 2$. Then, define $(G_1, f_1) \circ (G_2, f_2)$ to be the graph obtained from $G_1$ and $G_2$ by removing the pendant vertices $f_1, f_2$, and identifying the vertices $v_1$ and $v_2$.

\vskip 2mm \noindent
Then they proved:
\begin{theorem}\cite[Theorem 6.1]{BMS18}
Let $G$ be a bipartite graph. Then $J_G$ is Cohen-Macaulay if and only if $G = A_1 * A_2 * \cdots * A_k$, where $A_i = F_m$ or $A_i = F_{m_1} \circ \cdots \circ F_{m_r}$, for some $m \geq 1$ and $m_j \geq 3$.
\end{theorem}

\begin{figure}[H]
\begin{tikzpicture}
\draw [line width=1.pt] (1.,2.)-- (1.,1.);
\draw [line width=1.pt] (2.,2.)-- (1.,1.);
\draw [line width=1.pt] (2.,2.)-- (2.,1.);
\draw [line width=1.pt] (3.,2.)-- (2.,1.);
\draw [line width=1.pt] (3.,2.)-- (1.,1.);
\draw [line width=1.pt] (3.,2.)-- (3.,1.);
\draw [line width=1.pt] (4.,2.)-- (3.,1.);
\draw [line width=1.pt] (4.,2.)-- (4.,1.);
\draw [line width=1.pt] (4.,2.)-- (2.,1.);
\draw [line width=1.pt] (4.,2.)-- (1.,1.);
\begin{scriptsize}
\draw [fill=black] (1.,2.) circle (1.5pt);
\draw[color=black] (1.0898127435570215,2.2468642615417447) node {$1$};
\draw [fill=black] (1.,1.) circle (1.5pt);
\draw[color=black] (1.0898127435570215,0.8) node {$2$};
\draw [fill=black] (2.,2.) circle (1.5pt);
\draw[color=black] (2.1010175979349843,2.2468642615417447) node {$3$};
\draw [fill=black] (2.,1.) circle (1.5pt);
\draw[color=black] (2.1010175979349843,0.8) node {$4$};
\draw [fill=black] (3.,2.) circle (1.5pt);
\draw[color=black] (3.098739720921241,2.2468642615417447) node {$5$};
\draw [fill=black] (3.,1.) circle (1.5pt);
\draw[color=black] (3.098739720921241,0.8) node {$6$};
\draw [fill=black] (4.,2.) circle (1.5pt);
\draw[color=black] (4.096461843907497,2.2468642615417447) node {$7$};
\draw [fill=black] (4.,1.) circle (1.5pt);
\draw[color=black] (4.096461843907497,0.8) node {$8$};
\end{scriptsize}
\end{tikzpicture}
\caption{$F_4$}
\end{figure}
It follows from \cite[Theorem 3.1]{JNR} that for a Cohen-Macaulay bipartite graph $G = A_1 *\cdots *A_k$, $\reg(S/J_G) = \sum_{i=1}^k \reg(S/J_{A_i})$. Hence, to study the regularity, we only need to consider the graphs which are of the form of $A_i$ given in the statement of the previous theorem. In \cite{JA1}, we computed the regularity of Cohen-Macaulay bipartite graphs. In that paper, we observed an error in the statement and the proof of the regularity expression. These results were part of the doctoral thesis of Kumar. We had the corrected version in the thesis in 2019. Since that is not a public document, we take this opportunity to publish the corrected statement and a correct proof. We begin with setting up some notation for this purpose.

\begin{definition}
Let $W = \{v_1, \ldots , v_r\} \subseteq [n]$. Then, $F^W(K_n)$ is the graph obtained from $K_n$ by the following operation: for every $i = 1, \ldots ,r$,  attach a complete graph $K_{a_i}$ to $K_n$ in such a way that $V(K_n) \cap  V(K_{a_i}) = \{v_1, \ldots, v_i\}$, for some $a_i >	i$. We say that the graph $F^W(K_n)$ is obtained by adding a fan to $K_n$ on the set $W$, and $\{ K_{a_1},\ldots, K_{a_r}\} $ is the branch of that fan on  $W$.
	
Let $K_n$ be the complete graph on $[n]$ and $W_1 \sqcup \cdots \sqcup W_k$ be a partition of a subset $W \subseteq [n]$. Let $F_k^W(K_n)$ be the graph obtain from $K_n$ by adding a fan on each set $W_i$.  For each $i\in	\{1,\ldots , k \}$, set $W_i =\{v_{i,1},\ldots , v_{i,r_i}\}$ and 	$\{K_{a_{i,1}},\ldots , K_{a_{i,r_i}}\}$ be the branch of the fan on $W_i$. The graph $F_k^W(K_n)$ is called a $k$-fan of $K_n$ on the set $W$.
\end{definition}

\begin{figure}[H]
\begin{center}
\begin{tikzpicture}[scale=.6]
\draw (-1,4)-- (-0.94,2.68);
\draw (-0.94,2.68)-- (0.02,1.6);
\draw (0.02,1.6)-- (1.04,2.74);
\draw (1.04,2.74)-- (1,4);
\draw (1,4)-- (0,5);
\draw (0,5)-- (-1,4);
\draw (-1,4)-- (-2,5);
\draw (-1,4)-- (-1.16,5.66);
\draw (-1.16,5.66)-- (0,5);
\draw (0,5)-- (0,6);
\draw (0,6)-- (-1,4);
\draw (-1,4)-- (1,4);
\draw (1,4)-- (0,6);
\draw (0,5)-- (-0.94,2.68);
\draw (-0.94,2.68)-- (1.04,2.74);
\draw (1.04,2.74)-- (0,5);
\draw (-1,4)-- (0.02,1.6);
\draw (0.02,1.6)-- (1,4);
\draw (1,4)-- (-0.94,2.68);
\draw (-1,4)-- (1.04,2.74);
\draw (0,5)-- (0.02,1.6);
\draw (1.04,2.74)-- (2.14,2.92);
\draw (1.04,2.74)-- (1.34,1.38);
\draw (1.34,1.38)-- (0.02,1.6);
\draw (-2,5)-- (-2.28,3.96);
\draw (-2.28,3.96)-- (-1,4);
\draw (-1.52,5.08)-- (-1,4);
\draw (-1.52,5.08)-- (0,5);
\draw (-1.52,5.08)-- (-1.16,5.66);
\draw (1.04,2.74)-- (0.88,1.88);
\draw (0.88,1.88)-- (1.34,1.38);
\draw (0.88,1.88)-- (0.02,1.6);
\begin{scriptsize}
\fill  (0,5) circle (1.5pt);
\draw (0.32,5.22) node {$2$};
\fill  (-1,4) circle (1.5pt);
\draw (-1.22,3.74) node {$1$};
\fill  (1,4) circle (1.5pt);
\draw (1.54,4.24) node {$3$};				
\fill  (1.04,2.74) circle (1.5pt);
\draw (1.44,3.18) node {$4$};
\fill  (-0.94,2.68) circle (1.5pt);
\draw (-1.2,2.42) node {$6$};
\fill  (0.02,1.6) circle (1.5pt);
\draw (-0.08,1.26) node {$5$};
\fill  (-2,5) circle (1.5pt);
\fill  (-1.16,5.66) circle (1.5pt);
\fill  (0,6) circle (1.5pt);
\fill  (2.14,2.92) circle (1.5pt);
\fill  (1.34,1.38) circle (1.5pt);
\fill  (-2.28,3.96) circle (1.5pt);
\fill  (-1.52,5.08) circle (1.5pt);
\fill  (0.88,1.88) circle (1.5pt);
\end{scriptsize}
\end{tikzpicture}
\end{center}
\caption{A fan graph}\label{non-pure-fig}
\end{figure}
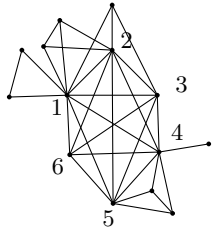

\begin{theorem}\cite[Theorem 4.6]{JA1}\label{4.6}
Let $n \geq 3$ and $H$ denote either $F_n$ or $F_k^W(K_n)$ with $W = W_1 \sqcup \cdots \sqcup W_k$ and $|W_i| \geq 2$ for some $i$. Let $G= F_{m_1} \circ \cdots \circ F_{m_t} \circ (H,f)$ be a graph with $t\geq1$, and for each $i\in [t]$, $m_i \geq 3$. Let $V(F_{m_1} \circ \cdots \circ F_{m_t}) \cap V(H)= \{v\}$ and $f$ be a pendant vertex in $N_H(v)$. If $H = F_k^W(K_n)$, then assume that $v\in W_i$ and $|W_i| \geq 2$. Then,  
$$\reg(S/J_G) = \reg(S/J_{F_{m_1} \circ \cdots \circ F_{m_{t-1}} \circ F_{m_t-1}})+\reg(S/J_{H\setminus \{v,f\}}).$$
\end{theorem}
\begin{proof}
For each $i \in \{1,\ldots,t\}$ and $j=\{1,2\}$, let $f_{i,j}$ be the pendant vertices of $F_{m_i}$ and for each $i \in\{1,\ldots,t-1\}$, $V(F_{m_i}) \cap V(F_{m_{i+1}})=\{v_{i,i+1}\}$, i.e., $F_{m_i} \circ F_{m_{i+1}}$ is the graph obtained from $F_{m_i}$ and $F_{m_{i+1}}$ by removing pendant vertices $f_{i,2},f_{i+1,1}$ and identifying vertices $2m_i-1 =v_{i,i+1} =2$. Assume, without loss of generality,  that $v \in W_1$. Following \cite[Lemma 4.8]{Oh11}, we have $J_G = J_{G_v} \cap ((x_v,y_v)+J_{G\setminus v})$. Consider the following short exact sequence: 
\begin{equation}\label{ses1}
0  \longrightarrow  \dfrac{S}{J_{G} } \longrightarrow \dfrac{S}{J_{G_v}} \oplus \dfrac{S}{(x_v,y_v)+J_{G \setminus v}} \longrightarrow \dfrac{S}{(x_v,y_v)+J_{G_v \setminus v} } \longrightarrow 0.
\end{equation}		
		
We  proceed by  induction  on $t$. If $t=1$, then the assertion follows from \cite[Propositions 4.3, 4.4]{JA1}. Now, assume that $t\geq2$, and that the result  is true for $ t-1$. Let $H=F_k^W(K_n)$ and $H_1$ be the complete graph on $N_G[v]$. Then, $G_v = F_{m_1} \circ \cdots \circ F_{m_{t-1}} \circ G_1$, where $G_1=F_k^{U}(H_1)$ is the $k$-pure fan  of $H_1$ on $U=N_{F_{m_t} \setminus f_{t,2}}(v) \sqcup (W \setminus W_1)$, $G_v \setminus v = F_{m_1} \circ \cdots \circ F_{m_{t-1}} \circ F_k^U(H_1\setminus v)$ and $G\setminus v = F_{m_1} \circ \cdots \circ F_{m_{t-1}} \circ F_{m_t-1} \sqcup H\setminus \{v,f\}$. Since, $|N_{F_{m_t} \setminus f_{t,2}}(v)| \geq 2$, by induction on $t$ and by \cite[Theorem 3.4]{JA1},
\begin{eqnarray*}
\reg(S/J_{G_v}) & =& \reg(S/J_{F_{m_1} \circ \cdots \circ F_{m_{t-1}-1}})+ \reg(S/J_{F_k^U(H_1)}) \\ &=& \reg(S/J_{F_{m_1} \circ \cdots \circ F_{m_{t-1}-1}})+k+1,\\
\reg(S/J_{G_v \setminus v}) & =&  \reg(S/J_{F_{m_1} \circ \cdots \circ F_{m_{t-1}-1}})+ \reg(S/J_{F_k^U(H_1 \setminus v)})\\ &=& \reg(S/J_{F_{m_1} \circ \cdots \circ F_{m_{t-1}-1}})+k+1.
\end{eqnarray*}
Note that $\reg(S/((x_v,y_v)+J_{G\setminus v})) = \reg(S/J_{F_{m_1} \circ \cdots \circ {F_{m_{t-1}}} \circ {F_{m_t-1}}})+ \reg(S/J_{H\setminus \{v,f\}}).$
			
If $m_t =3$, then $ F_{m_1} \circ \cdots \circ F_{m_{t-1}} \circ F_{m_t-1} = (F_{m_1} \circ \cdots \circ F_{m_{t-1}}) * F_{1} $. Therefore, by \cite[Theorem 3.1]{JNR},
\begin{eqnarray*}
\reg(S/J_{F_{m_1} \circ \cdots \circ F_{m_{t-1}} \circ F_{m_t-1}})& =  &\reg(S/J_{F_{m_1} \circ \cdots \circ {F_{m_{t-1}}}})+1\\ &>& \reg(S/J_{F_{m_1} \circ \cdots \circ F_{m_{t-1}-1}}),
\end{eqnarray*}
where the last inequality follows from \cite[Corollary 2.2]{MM}. If $m_t >3$, then by induction, 
\begin{eqnarray*}
\reg(S/J_{F_{m_1} \circ \cdots \circ F_{m_{t-1}} \circ F_{m_t-1}})& = &\reg(S/J_{F_{m_1} \circ \cdots \circ {F_{m_{t-1}-1}}})+\reg(S/J_{F_{m_t-2}})\\
&>& \reg(S/J_{F_{m_1} \circ \cdots \circ F_{m_{t-1}-1}}),	
\end{eqnarray*} where the last inequality follows from \cite[Proposition 4.1]{JA1}. Since $|W_1|\geq 2$ and $v\in W_1$, $H\setminus \{v,f\}$ is a $k$-pure fan graph. Therefore, by \cite[Theorem 3.4]{JA1}, $\reg(S/J_{H\setminus \{v,f\}})=k+1$. Using \cite[Lemma 2.2]{JA1} in the short exact sequence \eqref{ses1}, we get 
$$\reg(S/J_G)=\reg(S/J_{F_{m_1} \circ \cdots \circ {F_{m_{t-1}}} \circ {F_{m_t-1}}})+ \reg(S/J_{H\setminus \{v,f\}}).$$		
\indent	Now, assume that $H=F_n$. Let $H_2$ be the complete graph on  $N_G[v]$. Note that $G_v = F_{m_1} \circ \cdots \circ F_{m_{t-1}} \circ G_2$, where $G_2=F_2^{U'}(H_2)$ is the $2$-pure  fan  of $H_2$ on $U'=N_{F_{m_t} \setminus f_{t,2}}(v) \sqcup N_{F_{n} \setminus f}(v) $, $G\setminus v = F_{m_1} \circ \cdots \circ F_{m_{t-1}} \circ F_{m_t -1} \sqcup F_{n-1}$ and $G_v \setminus v = F_{m_1} \circ \cdots \circ F_{m_{t-1}} \circ F_2^{U'}(H_2\setminus	v)$.  Hence, by induction on $t$ and by \cite[Theorem 3.4]{JA1},
\begin{eqnarray*}
\reg(S/J_{G_v}) & = & \reg(S/J_{F_{m_1} \circ  \cdots \circ F_{m_{t-2}}\circ F_{m_{t-1} -1}})+\reg(S/J_{F_2^{U'}(H_2)})\\ & = & \reg(S/J_{F_{m_1} \circ  \cdots \circ F_{m_{t-2}}\circ
F_{m_{t-1} -1}})+3,\\
\reg(S/J_{G_v \setminus v}) & = & \reg(S/J_{F_{m_1} \circ  \cdots \circ F_{m_{t-2}}\circ F_{m_{t-1} -1}})+\reg(S/J_{F_2^{U'}(H_2\setminus v)})\\ & = & \reg(S/J_{F_{m_1} \circ  \cdots \circ F_{m_{t-2}}\circ F_{m_{t-1} -1}})+3.
\end{eqnarray*}
By \cite[Proposition 4.1]{JA1}, $\reg(S/J_{F_{n-1}})=3$. Note that 
\begin{eqnarray*}
\reg(S/((x_v,y_v)+J_{G\setminus v})) &= & \reg(S/J_{F_{m_1} \circ \cdots \circ {F_{m_{t-1}}} \circ {F_{m_t-1}}})+ \reg(S/J_{F_{n-1}})\\
&= &\reg(S/J_{F_{m_1} \circ \cdots \circ {F_{m_{t-1}}} \circ {F_{m_t-1}}})+ 3\\
&>& \reg(S/J_{G_v}) \geq \reg(S/((x_v,y_v)+J_{G_v\setminus v})).
\end{eqnarray*}
Using the short exact sequence \eqref{ses1} and \cite[Lemma 2.2]{JA1}, we conclude that
$$\reg(S/J_G)=  \reg(S/J_{F_{m_1} \circ \cdots \circ {F_{m_{t-1}}} \circ {F_{m_t-1}}}) + \reg(S/J_{F_{n-1}}).$$ Hence, the assertion follows.
\end{proof}
We would like to remark that, the earlier stated formula coincides
with the present one if $m_i \geq 4$ for $2 \leq i \leq t$. 

Now, we obtain a precise expression for the regularity  of Cohen-Macaulay bipartite graphs.
By \cite[Theorem 6.1]{BMS18}, if $G$ is a connected Cohen-Macaulay
bipartite graph, then there exists  $s \in \mathbb{N}$ such that
$G=G_1* \cdots* G_s$, where $G_i=F_{n_i}$ or $G_i=F_{m_{i,1}} \circ \cdots \circ F_{m_{i,t_i}}$, 
for some $n_i\geq1$, and $m_{i,j}\geq 3$, for each $j=1,\ldots,t_i$.
Due to \cite[Theorem 3.1]{JNR}, it is enough to compute the regularity
of connected indecomposable Cohen-Macaulay bipartite graphs. Assume,
without loss of generality,  that $G=F_{m_1} \circ \cdots \circ
F_{m_t}$, for $t \geq 2$,  and $m_i \geq 3$, for each $i$. Let $H_G$
be a graph on the vertex set $\{2,\ldots,t-1\}$ and edge set $\{
\{i,i+1\} \; : \; 2 \leq i \leq t-2, m_i =m_{i+1}=3\}.$ For a graph
$G$, let $\ma(G)$ denote the matching number of $G$.
\begin{theorem}\cite[Theorem 4.7]{JA1}\label{cm-bipartite}
Let $G=F_{m_1} \circ \cdots \circ F_{m_t}$. Then,
$$\reg(S/J_G) =2\alpha_G+2\ma(H_G)+t,$$ where $\alpha_G = |\{j
~ :\; 2 \leq j \leq t-1 \text{ and } m_{j} \geq 4  \}\sqcup \{1,t\}|$. 
\end{theorem}
\begin{proof}
We proceed by induction on $t$. If $t=2$, then $\alpha_G =2$ and
$\ma(H_G)=0$. Therefore, the assertion follows from \cite[Proposition
4.3]{JA1}. So,  assume that $t>2$. Note that $G= F_{m_1} \circ \cdots
\circ F_{m_{t-1}} \circ F_{m_t}$.  Thus, by Theorem \ref{4.6},
$$\reg(S/J_{G}) = \reg(S/J_{F_{m_1} \circ \cdots \circ
{F_{m_{t-1}-1}}}) + \reg(S/J_{F_{m_t-1}}).$$ By \cite[Proposition
4.1]{JA1}, $\reg(S/J_{F_{m_t-1}})=3$. Set $G'=F_{m_1} \circ \cdots
\circ {F_{m_{t-1}-1}}$. Therefore, it is enough to prove that
$\reg(S/J_{G'})= 2\alpha_G+2\ma(H_G)+t-3$.  Assume that $m_{t-1}>3$.
Then $\alpha_{G}=\alpha_{G'}+1$ and
$H_{G}=H_{G'} \sqcup \{t-1\}$. Therefore, $\ma(H_{G'})=\ma(H_G)$. By
induction,
$\reg(S/J_{G'})=2\alpha_{G'}+2\ma(H_{G'})+t-1=2\alpha_G+2\ma(H_G)+t-3$.
Now, assume that $m_{t-1}=3$. Then, $G'=F_{m_1} \circ \cdots \circ
F_{m_{t-2}} \circ {F_{2}}= (F_{m_1} \circ \cdots \circ {F_{m_{t-2}}})
* F_1$ is decomposable. Set $G''=F_{m_1} \circ \cdots \circ
{F_{m_{t-2}}}$. By \cite[Theorem 3.1]{JNR}, it is enough to  prove
that $\reg(S/J_{G''})= 2\alpha_G+2\ma(H_G)+t-4$. We have the following
cases:\\
\noindent  {\bf Case 1:} Suppose $m_{t-2} >3$. Then, $H_{G}=H_{G''}
\sqcup \{t-1\}$ and $\alpha_{G}=\alpha_{G''}-1$. Thus, $\ma(H_{G''})=\ma(H_G)$. Now, by induction,
$\reg(S/J_{G''})=2\alpha_{G''}+2\ma(H_{G''})+t-2=2\alpha_G+2\ma(H_G)+t-4$.\\
\noindent  {\bf Case 2:} Suppose $m_{t-2} =3$. Then, either
$H_{G}=H_{G''} \sqcup \{t-2,t-1\}$ or $H_{G}=H_{G''} *
\{t-3,t-2\} * \{t-2,t-1\}$. In both cases, $\ma(H_{G''})=\ma(H_G)-1$. Note that $\alpha_{G''}=\alpha_G$.  Thus, by induction, $\reg(S/J_{G''})=2\alpha_{G''}+2\ma(H_{G''})+t-2=2\alpha_G+2\ma(H_G)+t-4$. Hence, the assertion follows. 
\end{proof}
	
We illustrate the above result with the following example. 
\begin{example}
Let $G= F_3 \circ F_4 \circ F_3 \circ F_3 \circ F_3 $ be the graph shown below.
		
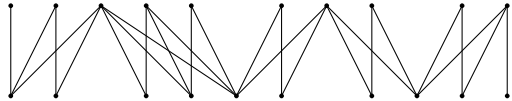
\begin{figure}[H]
\begin{center}
\begin{tikzpicture}[scale=.6]
\draw (-4,0)-- (-4,2);
\draw (-4,0)-- (-3,2);
\draw (-4,0)-- (-2,2);
\draw (-3,2)-- (-3,0);
\draw (-3,0)-- (-2,2);
\draw (-2,2)-- (-1,0);
\draw (-2,2)-- (0,0);
\draw (-1,0)-- (-1,2);
\draw (0,0)-- (0,2);
\draw (-1,2)-- (0,0);
\draw (-1,2)-- (1,0);
\draw (0,2)-- (1,0);
\draw (-2,2)-- (1,0);
\draw (1,0)-- (2,2);
\draw (1,0)-- (3,2);
\draw (2,2)-- (2,0);
\draw (2,0)-- (3,2);
\draw (3,2)-- (4,0);
\draw (3,2)-- (5,0);
\draw (4,2)-- (4,0);
\draw (4,2)-- (5,0);
\draw (5,0)-- (6,2);
\draw (5,0)-- (7,2);
\draw (6,2)-- (6,0);
\draw (7,2)-- (7,0);
\draw (7,2)-- (6,0);
\begin{scriptsize}
\fill  (-3,2) circle (1.5pt);
\fill  (-1,2) circle (1.5pt);
\fill  (-3,0) circle (1.5pt);
\fill  (-1,0) circle (1.5pt);
\fill  (-2,2) circle (1.5pt);
\fill  (0,0) circle (1.5pt);
\fill  (0,2) circle (1.5pt);
\fill  (1,0) circle (1.5pt);
\fill  (2,2) circle (1.5pt);
\fill  (2,0) circle (1.5pt);
\fill  (3,2) circle (1.5pt);
\fill  (4,2) circle (1.5pt);
\fill  (4,0) circle (1.5pt);
\fill  (5,0) circle (1.5pt);
\fill  (6,2) circle (1.5pt);
\fill  (6,0) circle (1.5pt);
\fill  (7,2) circle (1.5pt);
\fill  (7,0) circle (1.5pt);
\fill  (-4,2) circle (1.5pt);
\fill  (-4,0) circle (1.5pt);
\end{scriptsize}
\end{tikzpicture}
\end{center}
\end{figure}
Then $G$ is a Cohen-Macaulay bipartite graph with   $\alpha_G = 3$ and $\ma(H_G)= 1$. Therefore, by Theorem \ref{cm-bipartite}, $\reg(S/J_G) =6+2+5=13$.
\end{example}

All the results that we stated in this section are for binomial edge ideals. Therefore we would like to ask:

\begin{question}
Can one find a lower bound for precise expression for the generalized binomial edge ideals for all the classes of graphs discussed above?
\end{question}

In \cite{BMS22}, Bolognini et al. conjectured that for a graph $G$, $S/J_G$ is Cohen-Macaulay if and only if $G$ is ``accessible''. The property of being accessible is defined in terms of cut sets of a graph. A subset $T$ of $V(G)$ is said to be a cut set if the induced graph on $V(G) \setminus T$ has more components than the induced graph on $(V(G) \setminus T) \cup \{v\}$ for all $v  \in T$, i.e., every vertex $v$ in $T$ connects at least two distinct components of the induced graph on $V(G) \setminus T$. A graph $G$ is said to be \textit{accessible} if for every cut set $T$ of $G$, there exists a vertex $v \in T$ such that $T\setminus \{v\}$ is also a cut set. For example, if $G = C_n$ for some $n \geq 4$, is not accessible since two non-adjacent vertices of $C_n$ is  a cut set, but no singleton subset of $V(C_n)$ is a cut set. In \cite{BMS22}, the authors proved that if $G$ is a Cohen-Macaulay bipartite graph, then $G$ has a Hamiltonian path, i.e., a path that visits every vertex of a graph exactly once, \cite[Corollary 6.9]{BMS22}. A graph is said to be \textit{traceable} if it contains a Hamiltonian path. In view of \cite[Corollary 6.9]{BMS22} and \Cref{cm-bipartite}, Bolognini et al. raised the following question:
\begin{question}\cite[Problem 7.7]{BMS22}
Is it possible to find a formula or bounds for the regularity of Cohen-Macaulay binomial
edge ideals of traceable graphs?
\end{question}

We conclude our article by raising a question that has not been addressed in the literature so far:
\begin{question}
How is the regularity of $S/J_G$ and the regularity of $S/J_{K_m, G}$ are related? More precisely, if one knows $\reg(S/J_G)$, is there a way to compute or estimate $\reg(S/J_{K_m,G})$, at least for some important classes of graphs?
\end{question}


\bibliographystyle{amsplain}
\bibliography{Bibliography}

\end{document}